\documentclass{article}

\pdfoutput=1

\usepackage{amsmath}
\usepackage{amssymb}
\usepackage{caption}
\usepackage{amsfonts}
\usepackage{mathrsfs}
\usepackage{amsthm}
\usepackage{comment}
\usepackage{enumerate}
\usepackage{graphicx}
\usepackage{subfigure}
\usepackage{bbm}
\usepackage[numbers]{natbib}
\usepackage{tikz}
\usepackage{float}
\usepackage[utf8]{inputenc}
\usepackage[ruled,vlined]{algorithm2e}

\definecolor{rltred}{rgb}{0.75,0,0} 
	\definecolor{rltgreen}{rgb}{0,0.5,0}
	\definecolor{oneblue}{rgb}{0,0,0.75}
	\definecolor{marron}{rgb}{0.64,0.16,0.16}
	\definecolor{forestgreen}{rgb}{0.13,0.54,0.13}
	\definecolor{purple}{rgb}{0.62,0.12,0.94}
	\definecolor{dockerblue}{rgb}{0.11,0.56,0.98}
	\definecolor{freeblue}{rgb}{0.25,0.41,0.88}
	\definecolor{myblue}{rgb}{0,0.2,0.4}
	\definecolor{ccqqtt}{rgb}{0.8,0,0.2}

\numberwithin{equation}{section}

\newtheorem{step}{Step}
\newtheorem{lmm}{Lemma}
\newtheorem{rmrk}{Remark}
\newtheorem{prpstn}{Proposition}
\newtheorem{thrm}{Theorem}

\newcommand{\disp}{\displaystyle }
\def \k {{\kappa}}

\title{Exact simulation of Brownian diffusions with drift admitting jumps}

\author{David Dereudre\footnote{Laboratoire de Math\'ematiques Paul Painlev\'e,  UMR CNRS 8524, Universit\'e Lille1, 59655 Villeneuve d'Ascq Cedex, France. 
david.dereudre@univ-lille1.fr}, 
Sara Mazzonetto\footnote{Institut f\"{u}r Mathematik der Universit\"{a}t Potsdam. Karl-Liebknecht-Str. 24-25, 14476 Potsdam OT Golm, Germany, and 
Laboratoire de Math\'ematiques Paul Painlev\'e, UMR CNRS 8524, Universit\'e Lille1, 59655 Villeneuve d'Ascq Cedex, France. \mbox{mazzonet@uni-potsdam.de}}, 
Sylvie Roelly\footnote{Institut f\"{u}r Mathematik der Universit\"{a}t Potsdam. Karl-Liebknecht-Str. 24-25, 14476 Potsdam OT Golm, Germany, roelly@math.uni-potsdam.de}.}

\date{ }

\begin{document}

\maketitle

\textbf{Subject class}: {Primary 60H35, 65C05; Secondary 65C30, 68U20;}\\



\begin{abstract}
In this paper, using an algorithm based on retrospective rejection sampling scheme introduced in \cite{BRP,EM}, we propose an exact simulation of a Brownian diffusion whose drift admits several jumps. We treat explicitly and extensively the case of two jumps, providing numerical simulations. Our main contribution is to manage the technical difficulty due to the presence of {\it two} jumps thanks to a new explicit expression of the transition density of the skew Brownian motion with two semipermeable barriers and a constant drift.
\end{abstract}

\section{Introduction}

\subsection{The context}
The aim of the paper is to develop an exact simulation method for the real-valued Brownian diffusion defined on the finite time interval $[0,T]$ 
as solution of 
\begin{equation} \label{sdeb}
\begin{cases}
dX_t=dW_t+b(X_t)dt, \qquad t\in [0,T],\\
X_0=x_0,
\end{cases}
\end{equation}
where the drift function $b$ is a bounded regular map on ${\mathbb R}$ 
except on a finite set $J:=\{z_1,\ldots,z_n \}\subset {\mathbb R} $ where jumps occur.


In \cite{BR,BRR,BRP} first algorithms for simulating exactly the solution of (\ref{sdeb}) were provided in the case of an
everywhere regular drift $b$. 
The method has been recently improved in the papers \cite{EM, taylorth} where the authors treat the case of a drift with a unique jump ($n=1$).
We provide here a  non trivial generalization of  these results, proposing a theoretical exact simulation schemes in the case of a drift with several jumps ($n>1$) and treating explicitely the case of two discontinuities ($n=2$), at points $z_1$ and $z_2 $.

As we will explain in more details in the next section, our approach, which  is parallel to that of \cite{EM}, is based on a new explicit representation of the transition density of a skew Brownian motion with constant drift and two semipermeable barriers.

\subsection{The involved processes: definitions and notations}


Consider $\mathcal{C}:=\mathcal{C}([0,T],\mathbb{R})$ the canonical continuous path space and  $\mathscr{C}$ the Borel $\sigma$-algebra on $\mathcal{C}$ induced by the supremum norm. Let us fix $x_0\in\mathbb{R}$ and denote by $\mathbb{P}$ the Wiener measure on $(\mathcal{C}, \mathscr{C} )$, law of a Brownian motion $(W_t)_{t}$  starting from $x_0$. $(X_t)_{t}$ will always denote the canonical process.

Let us introduce the various processes which will be involved in our paper. We will always work under assumptions which assure their (weak) existence and uniqueness.\\
 $\mathbb{P}_b$ denotes the law on $\mathcal{C}$ of the Brownian motion with drift $b$, weak solution of ($\ref{sdeb}$). \\  
For $0\neq |\beta|\leq 1 $ let  $\mathbb{P}^{(\beta)}$ denote the law on $\mathcal{C}$ of the $\beta$-skew Brownian motion (SBM), solution of the stochastic differential equation
\[
\begin{cases}
\disp dX_t=dW_t+ \beta \, d L^{z}_t(X), \\
\disp X_0=x_0, \quad L^{z}_t=\int_0^t \mathbb{I}_{\{X_s=z\}} \, d L^{z}_s,
\end{cases}
\]
with $z\in\mathbb{R}$. The process was first introduced in \cite{IMcK} via a trajectorial definition, while its semimartingale properties were studied in \cite{HS}. 
The point $z$ is also called a  {\it semipermeable barrier} since the process is partially reflected at that point.
Notice that for $x_0>z$, $\mathbb{P}^{(1)} $ is the law of the reflected Brownian motion on the semiaxis $[z, +\infty[$, and for $x_0<z$, $\mathbb{P}^{(-1)} $ is a reflected Brownian motion on the semiaxis $]-\infty,z]$.\\
$\mathbb{P}^{(\beta)}_b$ will denote the law of the $\beta$-skew Brownian motion with drift $b$.\\
Analogously $\mathbb{P}^{(\beta_1,\beta_2)}$ (respectively $\mathbb{P}^{(\beta_1,\beta_2)}_b $) is the $(\beta_1,\beta_2)-$skew Brownian motion with two semipermeable barriers at $z_1$ and $z_2$, $z_1<z_2$, (resp. with drift $b$). 
We computed in \cite{DMR} an explicit expression for the transition probability density $p^{(\beta_1,\beta_2)}_\mu(t,x,y)$ of $\mathbb{P}^{(\beta_1,\beta_2)}_\mu$, the $(\beta_1,\beta_2)-$SBM with constant drift $\mu$ in the simpler special case when permeabilities and drift lead to the same trend ($\beta_1\mu>0,\beta_2\mu>0$).
In this paper we propose in Theorem \ref{th:tdf2skewdrift} an explicit expression for  $p^{(\beta_1,\beta_2)}_\mu(t,x,y)$ in full generality.\\
The $(\beta_1,\beta_2)-$SBM is a particular multi-skew Brownian motion, as defined in \cite{Ram}:
let $J=\{z_1,\ldots,z_n \}$ be an ordered set of $n$ barriers (or jumps) with respective  skewness coefficient $\beta_1, \cdots, \beta_n \in [-1,1]$. The $(\beta_1,\ldots,\beta_n)-$SBM with drift $b$, $\disp{\mathbb{P}_b^{(\beta_1,...,\beta_n)}}$, is the weak solution of
\begin{equation}\label{skewram}
\begin{cases}
\disp dX_t= dW_t +  b(X_t)\,dt + \sum_{j=1}^n \beta_j \, d L^{z_j}_t(X),\\
\disp X_0=x_0, \quad L^{z_j}_t(X)=\int_0^t \mathbb{I}_{\{X_s=z_j\}} \, dL^{z_j}_s , \quad j \in \{1,\cdots,n\}.
\end{cases}
\end{equation} 
See also \cite{Ouk},  \cite{LeGall1} or \cite{LLP} for the definition and the study of more general skew processes.

\subsection{Retrospective rejection sampling scheme for the exact simulation of $\mathbb{P}_b$} \label{EAscheme}
We now shortly recall the main idea introduced in \cite{BR} which allows to simulate exactly the law of the diffusion $\mathbb{P}_b$ on $\mathcal{C}$. 
The first key point is the following: it should exist another probability measure $\mathfrak{Q}$ on $\mathcal{C}$, called instrumental measure, such that 
\begin{itemize}
\item it is known how to sample from $\mathfrak{Q}$ (actually from its finite-dimensional distributions) 
\item  $\mathbb{P}_b$ is absolutely continuous with respect to this instrumental probability measure and
\begin{equation} \label{RNder}
\mathbb{P}_b(d X)\propto \disp e^{-\Phi_b(X)} \, \mathfrak{Q}(dX) \textrm{ where } \Phi_b(X):=\int_0^T\phi^+_b(X_t)\, dt .
\end{equation}
\end{itemize}
In other words, the log-density of $\mathbb{P}_b$ wrt $\mathfrak{Q}$  has the form of an additive functional, whose integrand $\phi^+_b$ is moreover supposed to be  positive and bounded (on $\mathbb{R}\setminus J $ in our framework). The proportionality sign $\propto$ indicates that there might be a renormalizing constant.\\
Once one has found an instrumental measure, it is possible to construct a rejection sampling scheme according to Proposition 1 in \cite{BR} and Theorem 1 in \cite{BRR}: given a sample  path $\omega$ from the instrumental measure, one needs an event $A$ with conditional probability with respect to $\omega$ equal to the Radon-Nikodym derivative $e^{-\Phi_b(\omega)}$. Then one accepts or rejects $\omega$ as a sample from the target measure depending on whether the event $A|\omega$ is satisfied or not.
Indeed, the desired event $A$ is obtained as follows: let $\Psi$ be an homogeneous Poisson process of unit intensity on the rectangle $[0,T]\times [0,\|\phi_b^+\|_\infty]$, then one sets $A=\{ \textit{all points of } \Psi \textit{ are above the graph of } t \mapsto \phi^+_b(\omega_t)\}$.\\
The resulting procedure to accept a given path is called \emph{retrospective} rejection method: first realize the random Poisson field $\Psi$, which leads to $M$ points $(\tau_1,x_1),\ldots,(\tau_M,x_M) \in [0,T]\times [0,\|\phi_b^+\|_\infty]$. Then, for each time $\tau_j$, sample $X_{\tau_j}$ under the law $\mathfrak{Q}$. Finally check if there is some point $(\tau_j,x_j)$ satisfying $x_j<\phi^+_b(X_{\tau_j})$. If this is not the case, accept $\left(\tau_j,x_j\right)_{j=1,\ldots,M}$, else start again. 
The algorithm returns a skeleton of a path $(X_t(\omega))_{t\in[0,T]}$ under $\mathbb{P}_b$:
$\left(X_{\tau_1}(\omega), X_{\tau_2}(\omega), \ldots, X_{\tau_{M}}(\omega), X_T(\omega)\right).$
The skeleton can be enlarged/completed adding the position of the process at any intermediary time  $t\in[0,T]$, following a bridge dynamics.\\

The main issue is to find in our context the appropriate instrumental measure $\mathfrak{Q}$ and the appropriate functional $\Phi_b$ to apply the retrospective rejection sampling scheme. The procedure is the following.
\begin{step}\label{Girs}
Apply Girsanov theorem to obtain the Radon-Nikodym derivative of $\mathbb{P}_b$ with respect to $\mathbb{P}$:
\[
\mathbb{P}_b(dX)=\exp{\Big(\int_0^T b(X_t)dX_t-\frac{1}{2}\int_0^T b^2(X_t)dt\Big)} \mathbb{P}(dX).
\]
To replace the It\^o stochastic integral $\disp \int_0^T b(X_t)dX_t$ by a Stieltjes one, introduce the function $\disp B(x):= \int_0^{x} b(y)dy$, primitive of $b$, and apply the It\^o-Tanaka formula
\[
B(X_T)-B(X_0)= \int_0^T \frac{\partial_+ B(X_t)+\partial_-B(X_t)}{2}dX_t 
+ \frac{1}{2}\int_{\mathbb{R}\setminus J} L_t^x(X) b'(x)dx 
+ \sum_{j=1}^n \theta_j L_T^{z_j}(X) 
\]
where $L_t^y$ is the symmetric local time in $y$ at time $t$ and $\theta_j$ is the (half) $j-th$ jump height of $b$,
\begin{equation} \label{deftheta}
\theta_j:=\frac{b(z_j^+)-b(z_j^-)}{2}, \quad j=1, \cdots,n .
\end{equation}
\end{step}
\begin{step} Thanks to the occupation time formula, obtain the decomposition $(\ref{RNder})$
\end{step}
\begin{equation} \label{Q}
\mathbb{P}_b(d X)\propto \exp{\left(-\Phi_b(X)\right)} \ 
\underbrace{ \exp{\bigg(B(X_T)-B(X_0)-\sum_{j=1}^n \theta_j L^{z_j}_T(X)\bigg)} 
\mathbb{P}(d X)}_{\propto \ \mathfrak{Q}(d X)}
\end{equation}
where $\disp \Phi_b(X):=\int_0^T\phi^+_b(X_t)\, dt$ with 
\begin{equation}\label{phib}
\phi_b^+(x):= \frac{1}{2}\left(b^2(x)+b'(x)-\inf_{y \in \mathbb{R}\setminus J }\left(b^2(y)+b'(y) \right)\right).
\end{equation}

The case of one discontinuity $(n=1)$ has been treated in \cite{EM},\cite{taylorth} and \cite{PRT}.
Nevertheless the method to sample from $\mathfrak{Q}$ presented in \cite{taylorth} differs from ours, and relies on the explicit expression of the joint distribution 
of the Brownian motion and its local time at one point. To our knowledge this latter technique is not generalizable to the case of a drift with several discontinuities because the joint distribution 
of the Brownian motion and its local time at several points is not yet explicitly known. 

\subsection{Outline of the paper}
In section \ref{exactsimscheme}, we describe the theoretical scheme for the exact simulation of $\mathbb{P}_b$.
The method consists in drawing a path from the instrumental measure $\mathfrak{Q}$, which is  interpreted as the weak limit of a sequence of measures which are themselves instrumental measures for the exact simulation scheme of some SBM with drift $b$. This limit interpretation was first presented in \cite{EM}. 

In section \ref{bound} we will complete concretely the exact simulation scheme for $\mathbb{P}_b$ in the particular case of two jumps ($n=2$).

The last section is devoted to the exact simulation pseudo-codes, and to the comparison with Euler-Maruyama method, which sometimes is faster than the exact one but is much less precise due to the effect of the discontinuities of the drift.

\section{Exact simulation for Brownian diffusions: the theory in our context}\label{exactsimscheme}

We are interested in the exact simulation of the Brownian diffusion $\mathbb{P}_b$, weak solution of equation (\ref{sdeb}).\\
Let us recall our assumptions on the drift $b$: it is a bounded function of class $C^1$ with bounded derivative on $\mathbb{R}\setminus J$. At each of the $n$ discontinuity points $z_1,\ldots,z_n \in J$,  $b$ admits a jump with (half) height given by $\left(\theta_j\right)_{j=1,\ldots,n}$ defined in (\ref{deftheta}). Since only the left and right limit values $b(z_i^-)$ and $b(z_i^+)$ (and not $b(z_i)$) do matter, for simplicity, we will suppose that the value of the function $b$ at each point $z_i$ coincides with the average value
\begin{equation}\label{bzi}
b(z_i):=\frac{b(z_i^+)+b(z_i^-)}{2}, \quad z_i \in J.
\end{equation}
The primitive of $b$, called $B$, is well defined since $b$ is locally integrable. It is the difference of convex functions because $b$ has bounded variation. The latter property will be required to apply the It\^o-Tanaka formula (see e.g. \cite{kalle}, Theorem 22.5).
At the points of discontinuity,  $b'$ can be arbitrarily defined, for example as
$$
b'(z_i):= \frac{\partial_+ b(z_i) +\partial_- b(z_i)}{2}, \quad z_i \in J.
$$
Therefore  $\phi^+_b$ defined in $(\ref{phib})$ is non negative and bounded.

The boundedness of $b$  implies the existence and uniqueness of a weak solution for (\ref{sdeb}) (see \cite{SV}), even the existence and uniqueness of a strong solution thanks to Zvonkin and Yamada-Watanabe results, see \cite{YW}. 

The aim of this section is to propose a theoretical method to provide exact samples from the instrumental measure $\mathfrak{Q}$ defined in (\ref{Q}).
We generalize the scheme proposed in \cite{EM}, to our situation with several jumps. 
Indeed we are able to  find $\beta_1,\ldots,\beta_n\in [-1,1]$ 
which allow to develop the scheme presented in Section \ref{EAscheme}
\begin{enumerate}
\item there exists an instrumental measure $\mathfrak{Q}^{(\beta_1,\dots,\beta_n)}$ such that 
\[ 
\mathbb{P}_b^{(\beta_1,\ldots,\beta_n)}(d X)\propto \exp{\left(-\Phi_b(X)\right)} \ \mathfrak{Q}^{(\beta_1,\ldots,\beta_n)}(d X) \] 
where $\mathbb{P}_b^{(\beta_1,\ldots,\beta_n)}$ is the SBM with drift defined by \eqref{skewram};
\item the convergence of  $\mathbb{P}_b^{(\beta_1,\ldots,\beta_n)}$ to the Brownian motion with drift, $\mathbb{P}_b$ 
, when the $\beta_j$'s tend to $0$, implies the convergence of the instrumental measures $\mathfrak{Q}^{(\beta_1,\ldots,\beta_n)}$ to $\mathfrak{Q}$;
\item the exact simulation scheme {\em propagates} at the limit:
\[ \begin{split} 
\mathbb{P}^{(\beta_1,\ldots,\beta_n)}_b(d X) & \propto \exp{\left(-\Phi_b(X)\right)} \ \mathfrak{Q}^{(\beta_1,\ldots,\beta_n)} (d X) \\
 \beta_j \to 0  \downarrow  \qquad  &  \qquad \qquad \qquad \qquad \quad  \downarrow\\
\mathbb{P}_b(d X) & \propto \exp{\left(-\Phi_b(X)\right)} \ \mathfrak{Q}(d X)
\end{split}\]
(see Section~\ref{limitalgorithm}).
\end{enumerate}

\subsection{On the exact simulation of the skew Brownian motion with drift} \label{ESMbetab}
 For the purpose of exact simulation of $\mathbb{P}_b$, it is sufficient to provide an exact simulation scheme for $\mathbb{P}_b^{(\beta_1,\ldots,\beta_n)}$ with small skewness parameters $(\beta_j)_{j=1,\ldots,n}$.
The algorithm is theoretical because it relies on knowing an explicit expression of the transition densities $p^{(\beta_1,\ldots,\beta_n)}_\mu$ for a constant (or piecewise constant) $\mu$. The method introduced in \cite{DMR,LLP} (and improved in Theorem~\ref{th:tdf2skewdrift} in Section~\ref{Sec:3.1} for $n=2$) can be exploited to find it.

\subsubsection{The instrumental measure}
One knows by Girsanov theorem that 
\begin{equation} \label{itoskew}
\mathbb{P}^{(\beta_1,\ldots,\beta_n)}_b(dX) \propto \exp{\left(-\Phi_{b}(X)\right)} \exp{\Big(B(X_T)-B(X_0)- \sum_{j=1}^n \beta_j C_{\beta_j} L_T^{z_j}(X)\Big)} \ \mathbb{P}^{(\beta_1,\ldots,\beta_n)}(dX),
\end{equation}
where $C_{\beta_j}=b(z_j)+\frac{\theta_j}{\beta_j}$, and $b(z_j)$ is defined in (\ref{bzi}).\\
Except for the precise values  $\beta_j=-\frac{\theta_j}{b(z_j)}, j=1,\ldots,n$ for which $C_{\beta_j}=0$, the local time terms in \eqref{itoskew} do not vanish, which makes the simulation hard. To get around this difficulty, one 
looks for a real number $\mu$ such that 
\begin{equation} \label{RNderbeta}
\mathbb{P}_b^{(\beta_1,\ldots,\beta_n)}(dX) \propto \exp{\left(-\Phi_{b}(X)\right)} \ 
\underbrace{  \exp{\left(B^{(\mu)}(X_T)-B^{(\mu)}(X_0)\right)}\  \mathbb{P}^{(\beta_1,\ldots,\beta_n)}_\mu(dX),}
_{\propto \ \mathfrak{Q}^{(\beta_1,\ldots,\beta_n)} (dX)} 
\end{equation}
where $B^{(\mu)}(x):=B(x) - \mu x$ is a primitive of $b-\mu$. 
Thus, knowing how to sample under the measure $\mathfrak{Q}^{(\beta_1,\ldots,\beta_n)}$, it would be possible to simulate exactly $\mathbb{P}_b^{(\beta_1,\ldots,\beta_n)}$ following the method described at the beginning of Section \ref{EAscheme}.

In the next lemma, we show how to choose $\beta_1,\ldots,\beta_n \in [-1,1]$ and $\mu \in \mathbb{R}$ in such a way that \eqref{RNderbeta} holds.


\begin{lmm}\label{choicebeta}
The representation \eqref{RNderbeta} holds as soon as the skewness parameters and the constant $\mu$ satisfy
\begin{equation} \label{betaj}
\beta_j = \beta_1 \frac{\theta_j} {\theta_1 + \beta_1 \big(b(z_1)-b(z_j)\big)}, j=2,\ldots,n, \quad
\textrm{ and } \quad \mu=b(z_1)+\frac{\theta_1}{\beta_1}.
\end{equation}
\end{lmm}
\begin{proof} Girsanov theorem yields
\[
\mathbb{P}^{(\beta_1,\ldots,\beta_n)}_b(dX) \propto\exp{\bigg(B^{(\mu)}(X_T)-B^{(\mu)}(X_0)-\Phi_b(X)- \sum_{j=1}^n  \beta_j  \Big( C_{\beta_j} -\mu \Big) L_T^{z_j}(X)\bigg)} \ 
\mathbb{P}_\mu^{(\beta_1,\ldots,\beta_n)}(dX).
\]
To erase the coefficients in front of  the local times, it is sufficient  to set $\mu=C_{\beta_1}=\cdots=C_{\beta_n}$, which leads to the identities (\ref{betaj}).
\end{proof}

\begin{rmrk} \label{rem:limitbetapetit}
For our purpose, we are indeed free to let depend the coefficients $(\beta_j)_{j=2,\ldots,n}$ from $\beta_1$, as long as  $\lim_{\beta_1\to 0} \beta_j=0$ for $j=2,\ldots,n$. 
Notice that, for the above choice of $\beta_1,\ldots,\beta_n$ and $\mu$, since $\beta_j\mu=b(z_j)\beta_j+\theta_j$ one has $\lim_{\beta_1\to 0} \beta_j \mu=\theta_j$ for any $j=2,\ldots,n$. The latter limit will have some importance in the limit procedure we will apply next.
\end{rmrk}

\begin{rmrk}[Retrospective rejection sampling for $\mathbb{P}_b^{(\beta_1,\ldots,\beta_n)}$] \label{piecewisemu}
Let $\beta_1, \ldots, \beta_n \subset [-1,1]$ not necessarily small. Assume there exist two indices $j_1,j_2$ such that $C_{\beta_{j_1}}\neq C_{\beta_{j_2}}$ (and they do not vanish).
Identity \eqref{RNderbeta} still holds if one replaces the constant drift $\mu$ by a well chosen piecewise constant function. Therefore for the purpose of the retrospective rejection sampling one needs an explicit expression for the transition density of the skew Brownian motion with piecewise constant drift. It is a straightforward generalization of the results presented here.
An integral representation for the transition density of the one-skew Brownian motion with two-valued drift $\mu$ has been already given in \cite{LLP}. 
\end{rmrk}

\subsubsection{Simulation of $\mathfrak{Q}^{(\beta_1,\ldots,\beta_n)}$} \label{simQ}

To simulate $\mathfrak{Q}^{(\beta_1,\ldots,\beta_n)} (dX)$ means to sample from its finite dimensional marginals, that is to 
give a finite number of sample variates - called  skeleton $(X(t_1),\ldots, X(t_{M}), X_T)=(y_1, \ldots, y_{M},y)$ -  from the density
\begin{equation} \label{sklaw}
h^{\left(\beta_1,\ldots,\beta_n\right)}(y)\prod_{i=0}^{M-1} q^{\left(\beta_1,\beta_2,\ldots,\beta_n\right)}_{\mu}(t_{i+1}-t_i,T-t_i,y_i,y,y_{i+1})dy_1\ldots dy_{M}dy,
\end{equation}
where $t_0:=0$, $y_0:=X_0=x_0$ and $q^{(\beta_1,\ldots,\beta_n)}_\mu$ is the transition density of the bridge of $\mathbb{P}^{(\beta_1,\ldots,\beta_n)}_\mu$. 
The density function $h^{(\beta_1,\ldots,\beta_n)}(y)$ of the conditioned law of $X_T|{X_0=x_0}$ satisfies 
\[
h^{(\beta_1,\ldots,\beta_n)}(y)\propto \exp{\Big(B^{(\mu)}(y)-B^{(\mu)}(x_0)\Big)} p^{(\beta_1,\ldots,\beta_n)}_\mu(T,x_0,y),
\]
where $p^{(\beta_1,\ldots,\beta_n)}_\mu$ is the transition density of $\mathbb{P}^{(\beta_1,\ldots,\beta_n)}_\mu$, as already introduced.\\
The relationship between  $q^{(\beta_1,\ldots,\beta_n)}_\mu$ and $p^{(\beta_1,\ldots,\beta_n)}_\mu$ is, as usual for bridges, given by
\[\begin{split}
q^{(\beta_1,\ldots,\beta_n)}_\mu(t,T,x_1,x_2,y) & = \frac{p^{(\beta_1,\ldots,\beta_n)}_\mu(t,x_1,y) \ p^{(\beta_1,\ldots,\beta_n)}_\mu(T-t,y,x_2)}{p^{(\beta_1,\ldots,\beta_n)}_\mu(T,x_1,x_2)}.
\end{split}\]
Let us write for simplicity $p_\mu$ instead of $p^{(0,\cdots,0)}_\mu$ (resp. $q_\mu$ instead of $q^{(0, \cdots,0)}_\mu$) for the transition density of a Brownian motion with constant drift $\mu$ (resp. for the transition density of a Brownian bridge with constant drift $\mu$).
Recall also that $q_\mu$ does not depend on the drift $\mu$ and is equal to $q_0$, the transition density of the Brownian bridge.\\
The challenge is therefore to obtain an explicit expression for the transition density $p^{(\beta_1,\ldots,\beta_n)}_\mu$ involving an instrumental density from which it is known how to sample, namely $p_\mu$ the transition density of the Brownian motion with drift $\mu$
\begin{equation} \label{decomposition}
p^{(\beta_1,\ldots,\beta_n)}_\mu(t,x,y) =p_\mu(t,x,y) \ v^{(\beta_1,\ldots,\beta_n)}_\mu(t,x,y),
\end{equation}
The function $v^{(\beta_1,\ldots,\beta_n)}_\mu(t,x,y)$ has to be uniformly bounded as a function of $y$ (and as a function of $x$).
If such a decomposition exists the densities involved in (\ref{sklaw}) are given by
\begin{equation} \label{qbeta}
q^{(\beta_1,\ldots,\beta_n)}_\mu(t,T,x_1,x_2,y)
 = q_0(t,T,x_1,x_2,y) \frac{v^{(\beta_1,\ldots,\beta_n)}_\mu(t,x_1,y)v^{(\beta_1,\ldots,\beta_n)}_\mu(T-t,y,x_2)}{v^{(\beta_1,\ldots,\beta_n)}_\mu(T,x_1,x_2)},
\end{equation}
\[\textrm{and } \quad h^{(\beta_1,\ldots,\beta_n)}(y) \propto \exp{\left(B^{(\mu)}(y)-B^{(\mu)}(x_0)\right)} \ p_\mu (T,x_0,y) \ v^{(\beta_1,\ldots,\beta_n)}_\mu(T,x_0,y).\]
One may then apply the generalized rejection sampling method thanks to the uniform boundedness of the function $v^{(\beta_1,\ldots,\beta_n)}_\mu$.
This important property  also implies the integrability of $h^{(\beta_1,\ldots,\beta_n)}$, as we prove in the next lemma.
\begin{lmm}\label{hintegrable} Let $x_0\in\mathbb{R}$. Under our assumptions on the drift $b$,
for any choice of $\beta_1,\ldots,\beta_n\in[-1,1]$ and  $\mu\in \mathbb{R}$, 
the density function $y\mapsto h^{\left(\beta_1,\ldots,\beta_n\right)}(y)$ is integrable and therefore normalizable.
\end{lmm}
\begin{proof}
The density $h^{(\beta_1,\ldots,\beta_n)}(y)$ satisfies
\begin{equation} \label{hbeta}
h^{(\beta_1,\ldots,\beta_n)}(y) \propto e^{B(y)-B(x_0)} \ p_0(T,x_0,y)  \ v^{(\beta_1,\ldots,\beta_n)}_\mu(T,x_0,y).
\end{equation}
Since $b$ is bounded, its primitive $B$ satisfies  $B(y)-B(x_0) \leq \|b\|_{\infty} |y-x_0|$, and 
\[\begin{split}
\int_{\mathbb{R}} \,  e^{B(y)-B(x)} \ p_0(T,x,y) \ v^{(\beta_1,\ldots,\beta_n)}_\mu(T,x,y) \ dy 
& \disp \leq \frac{ \|v^{(\beta_1,\ldots,\beta_n)}_\mu(T,x_0,\cdot)\|_\infty}{\sqrt{2\pi T}} \int_{\mathbb{R}} \,  e^{ \|b\|_{\infty}|y-x_0| } \ e^{-\frac{(y-x_0)^2}{2 T}}  \ dy \\
& \disp =  
 \|v^{(\beta_1,\ldots,\beta_n)}_\mu(T,x_0,\cdot)\|_\infty \ e^{\frac{\|b\|_{\infty}^2}{2} T}.
\end{split}\]
The integrability follows.
\end{proof}

\begin{rmrk}
$i)$ Our  decomposition \eqref{decomposition} is inspired by the one proposed in \cite{EMloc} in the simpler framework of a one-skew Brownian motion with constant drift $\mu$. There it is used to complete an exact simulation scheme for a diffusion whose  drift admits only one discontinuity.\\
$ ii)$ 
The tools provided in this paper yield indeed
the exact simulation of $\mathbb{P}^{(\beta_1,\beta_2)}_b$ with drift $b$ admitting two discontinuities, and for any parameters $\beta_1,\beta_2$ satisfying \eqref{betaj}. Notice that the skewness $\beta_i$ is not necessarily small.
\end{rmrk}

\subsection{The exact simulation scheme of $\mathbb{P}_b$ as a limit scheme} \label{limitalgorithm}

In the entire section, we correlate the parameters $(\beta_j)_{j=1,\ldots,n}$ and $\mu$ as in Lemma \ref{choicebeta}. The following convergence result
provides a method for sampling under the measure $\mathfrak{Q}$, once $p^{(\beta_1,\ldots,\beta_n)}_\mu$ is made explicit and uniformly bounded. 
This result yields the completion of the simulation scheme for $\mathbb{P}_b$.
\begin{lmm} \label{weakconv}
Take $\beta_1=\frac{1}{\k}$ and define the other skewness parameters $\beta_j(\k), j=2,\ldots,n$, by the relationship $(\ref{betaj})$: 
$\disp \beta_j (\k) = \frac{\theta_j} {\kappa \theta_1 + b(z_1)-b(z_j)}$  and $\mu(\k)=b(z_1)+\kappa \theta_1$. Then
\[
\mathbb{P}_{b}^{\left(\frac{1}{\k},\beta_2(\k),\ldots,\beta_n(\k)\right)} {\underset{\k\to\infty}
\longrightarrow} \mathbb{P}_b \quad \textrm{ and }
 \quad  \mathfrak{Q}^{\left(\frac{1}{\k}, \beta_2(\k),\ldots,\beta_n(\k)\right)} {\underset{\k\to\infty}\longrightarrow} \mathfrak{Q},
\]
where the convergences hold in the weak topology.
\end{lmm}
\begin{proof} First notice that even a stronger convergence holds in the following sense:
If $Y$ is the strong solution of equation (\ref{sdeb}) and $Y^{(\k)}, \k\geq 1,$ is the strong solution  
of equation (\ref{skewram}) with skewness coefficients $\beta_1=\frac{1}{\k}, \beta_2(\k), \cdots, \beta_n(\k)$, then the sequence $(Y^{(\k)})_\k$ converges in $L^1(\mathcal{C})$ to $Y$:
\[
\lim_{\k\to\infty}\mathbb{E}\Big(\sup_{s \in [0,T]}|Y^{(\k)}_s-Y_s|\Big)=0.
\]
This is due to Theorem 3.1 in \cite{LeGall1}, slightly extended in  \cite{EM}, Theorem 5.1. 
Strong convergence then implies the weak convergence of $(\mathbb{P}_{b}^{\left(\frac{1}{\k},\beta_2(\k),\ldots,\beta_n(\k)\right)} )_\k $ towards $\mathbb{P}_b$. \\
The weak convergence of the sequence $\mathfrak{Q}^{\left(\frac{1}{\k}, \beta_2(\k),\ldots,\beta_n(\k)\right)}$ is straightforward once one has noticed that the function defined on $\mathcal{C}$ by  $X \mapsto \exp{\left(\Phi_{b}(X)\right)}$ is bounded and continuous in the topology of the sup-norm. 
\end{proof}

As consequence of this lemma one obtains that a retrospective rejection sampling for $\mathbb{P}_b$ is possible as soon as it is possible to sample from $\mathfrak{Q}$, as explained at the beginning of this section. 
Moreover Lemma \ref{weakconv} yields also the sampling method for the instrumental measure. Indeed the weak convergence of $\mathfrak{Q}^{\left(\frac{1}{\k},\beta_2(\k),\ldots,\beta_n(\k)\right)}$ ensures the convergence of the finite dimensional distributions with density given by (\ref{sklaw}) to the finite dimensional marginals 
of $\mathfrak{Q}$. 
Therefore, if the sequence of densities $(h^{\left(\frac{1}{\k},\beta_2(\k),\ldots,\beta_n(\k)\right)})_\kappa$ and  $(q^{\left(\frac{1}{\k},\beta_2(\k),\ldots,\beta_n(\k)\right)}(t,T,x_1,x_2,\cdot))_\kappa$ admit pointwise limits $h^{(\theta_1,\dots,\theta_n)}$ and $q^{(\theta_1,\ldots,\theta_n)}(t,T,x_1,x_2,\cdot)$ respectively (see (\ref{qbeta}) and (\ref{hbeta})), then the limit of the expression (\ref{sklaw}) is given by
\begin{equation} \label{limitsklaw}
h^{(\theta_1,\ldots,\theta_n)}(y)\prod_{i=0}^{M-1} q^{(\theta_1,\ldots,\theta_n)}(t_{i+1}-t_i,T-t_i,y_i,y,y_{i+1})\, dy_1\ldots dy_{M}dy,
\end{equation}
which is the density of $(X_{t_1}, X_{t_2},\ldots, X_{t_{M}}, X_T)$ under $\mathfrak{Q}$.

We focus directly on the limits  $h^{(\theta_1,\ldots,\theta_n)}(y)d y$ and $q^{(\theta_1,\ldots,\theta_n)}(t,T,x_1,x_2,y)d y$, providing a rejection sampling scheme for them with instrumental densities respectively the transition density of the Brownian motion $p_0 (\frac{T}{1-\delta},x_0,y)$ and the transition density of the Brownian bridge $q_0 (t,T,x_1,x_2,y)$. \\
It is necessary to find positive functions
$f^{\mathcal{H}}_{\delta}$ and $f^{\mathcal{B}}_{x_1,x_2}$ smaller than $1$ such that
\begin{equation} \label{fractionrejection}
\frac{h^{(\theta_1,\theta_2,\ldots,\theta_n)}(y)}{p_0 (\frac{T}{1-\delta},x_0,y)}=C^{\mathcal{H}} \cdot f^{\mathcal{H}}_{\delta}(y), \quad \textrm{ and }
\frac{q^{(\theta_1,\theta_2,\ldots,\theta_n)}(t,T,x_1,x_2,y)}{q_0 (t,T,x_1,x_2,y)}= C^{\mathcal{B}} \cdot f^{\mathcal{B}}_{x_1,x_2}(y).
\end{equation}
 The parameter $\delta\in(0,1)$ can be chosen in an appropriate way, see \eqref{GRSh}.
The normalizing constant $C^{\mathcal{H}}$ (resp. $C^{\mathcal{B}}$) does only depend on $(\theta_1,\ldots,\theta_n), T, \ \|b\|_{\infty}, \ \delta$ (resp. on $(\theta_1,\ldots,\theta_n), t, \ T, \ x_1, \ x_2$).\\
Assuming that the decomposition (\ref{decomposition}) holds and assuming that there exists a pointwise limit $v^{(\theta_1,\ldots,\theta_n)}$ for $v^{(\frac{1}{\k},\beta_2(\k),\ldots,\beta_n(\k))}_{\mu(\k)}$ when $\kappa$ tends to $\infty$, one can pass to the limit in the equations (\ref{qbeta} and \ref{hbeta}) obtaining the relationships
\begin{equation} \label{qtheta}
\begin{split} q^{(\theta_1,\ldots,\theta_n)}(t,T,x_1,x_2,y)
& = q_0(t,T,x_1,x_2,y) \,\frac{v^{(\theta_1,\ldots,\theta_n)}(t,x_1,y)v^{(\theta_1,\ldots,\theta_n)}(T-t,y,x_2)}{v^{(\theta_1,\ldots,\theta_n)}(T,x_1,x_2)},\\
h^{(\theta_1,\ldots,\theta_n)}(y) & \propto \exp{\left(B(y)-B(x_0)\right)} \ p_0(T,x_0,y) \ v^{(\theta_1,\ldots,\theta_n)}(T,x_0,y).
\end{split}
\end{equation}
It is then sufficient to prove the existence of a uniform bound for $(x,y)\mapsto v^{(\theta_1,\ldots,\theta_n)}(t,x,y)$ in order to develop the rejection sampling scheme from $h^{(\theta_1,\ldots,\theta_n)}$ and $q^{(\theta_1,\ldots,\theta_n)}$. Moreover this bound yields the integrability of $h$, analogously to Lemma \ref{hintegrable}.\\
In conclusion, Lemma \ref{weakconv} ensures that the exact simulation scheme for the skew Brownian motion with drift $b$ is transposed to the desired exact simulation scheme for the Brownian diffusion with drift $b$. 


\section{Main tool: a uniform bound for the transition probability density} \label{bound}

From now on we will consider the case of a drift with two discontinuities ($n=2$) at the points $0$ and $z>0$, that is $J=\{0,z\}$. 
Our aim is to make explicit (see Theorem \ref{th:tdf2skewdrift}) and to control an expression of the form given in (\ref{decomposition}) for the transition density of the skew Brownian motion with constant drift $\mu$ and two semipermeable barriers at $0$ and $z$. 
We are actually interested in bounding this expression for vanishing skewness parameters,
see Remark~\ref{rem:limitbetapetit}. This is the content of Proposition \ref{vthetalimit} and Proposition ~\ref{boundtheta}.


\subsection{The transition density of the $(\beta_1,\beta_2)$-SBM with constant drift and its  limit} \label{Sec:3.1}

Let us define four non negative functions
\begin{equation}\label{aj}
\begin{cases}
a_1(x,y)\equiv 0\\
a_2(x,y)=|x|+|y|-|y-x|\\
a_3(x,y)=|x-z|+|y-z|-|y-x|\\
a_4(x,y)=2\left(z-\max(x,y,0)\right)^+ + 2 \min(x,y,z)^+.
\end{cases}\end{equation}

\begin{figure}[H] 
\begin{center}
\begin{tikzpicture}[scale=0.8]
\draw[draw] (0,0) -- (6,0);
\draw[dashed] (0,0) -- (-0.5,0) node[left]{Case $i)$};
\draw[dashed] (6,0) -- (6.5,0);
\draw[draw] (2,0.2) -- (2,-0.2) node[below]{$0$};
\draw[draw] (4,0.2) -- (4,-0.2) node[below]{$z$};
\fill[blue] (1,0) circle (0.1) node[above]{$x\wedge y$};
\fill[magenta] (5,0) circle (0.1) node[above]{$x \vee y$};
\end{tikzpicture}\hspace{2cm}
\begin{tikzpicture}[scale=0.8]
\draw[draw] (0,0) -- (6,0);
\draw[dashed] (0,0) -- (-0.5,0) node[left]{Case $ii)$};
\draw[dashed] (6,0) -- (6.5,0);
\draw[draw] (2,0.2) -- (2,-0.2) node[below]{$0$};
\draw[draw] (4,0.2) -- (4,-0.2) node[below]{$z$};
\fill[blue] (1,0) circle (0.1) node[above]{$x\wedge y$};
\fill[magenta] (3,0) circle (0.1) node[above]{$x \vee y$};
\end{tikzpicture}\hspace{2cm}
\begin{tikzpicture}[scale=0.8]
\draw[draw] (-1,0) -- (6,0);
\draw[dashed] (-1,0) -- (-1.5,0) node[left]{Case $iii)$};
\draw[dashed] (6,0) -- (6.5,0);
\draw[draw] (2,0.2) -- (2,-0.2) node[below]{$0$};
\draw[draw] (4,0.2) -- (4,-0.2) node[below]{$z$};
\fill[blue] (0.1,0) circle (0.1) node[above]{$x \wedge y$};
\fill[magenta] (1.5,0) circle (0.1) node[above]{$x \vee y$};
\end{tikzpicture}
\captionsetup{singlelinecheck=off}
\caption[contour]{
The different values of $a_j, j=1,2,3,4, $ according to the relative positions of $x,y$ and $z$:
\begin{itemize}
\item[Case $i)$]
For all $j=1,2,3,4$, $a_j=0$ 
\item[Case $ii)$] 
$a_1=a_2=0$ and $a_3=a_4= 2 \,(z- x\vee y)$
\item[Case $iii)$] 
$a_1=0,a_2=-2 \, x \vee y, a_3=2 \,(z-x\vee y), a_4=2 z$.
\end{itemize}
}\label{fig:variouscases}
\end{center}
\end{figure}
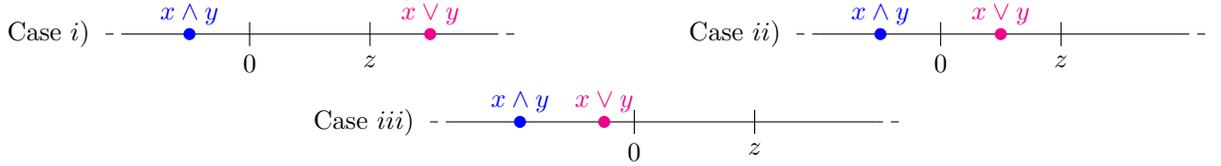

Let us also introduces polynomials of second order in $w$ setting  
$$
c_j(y,\mu;w):=w^2 \, c_{j,0}(y)+w \, \mu c_{j,1}(y)+ \mu^2 c_{j,2}(y), j \in \{1,2,3,4\},
$$
 where 
\begin{equation}\label{eq:cj}
\begin{cases}
c_{1,0}(y)=1, 											\\
c_{2,0}(y)=\left(2\mathbbm{1}_{\{y>0\}}-1\right)\beta_1 							\\
c_{3,0}(y)=\left(2\mathbbm{1}_{\{y>z\}}-1\right)\beta_2 							\\
c_{4,0}(y)=\left(1-2\mathbbm{1}_{[0,z)}(y)\right)\beta_1\beta_2 
\end{cases},
\begin{cases}
c_{1,1}(y)=\beta_1+\beta_2		\\
c_{2,1}(y)=-\beta_1-c_{4,0}(y)		\\
c_{3,1}(y)=-\beta_2+c_{4,0}(y)		\\
c_{4,1}(y)=0
\end{cases},
\begin{cases}
c_{1,2}(y)=\beta_1\beta_2 \\
c_{2,2}(y)=\beta_1 c_{3,0}(y) \\
c_{3,2}(y)=-\beta_2 c_{2,0}(y) \\
c_{4,2}(y)=-c_{4,0}(y).
\end{cases}
\end{equation}
The polynomials $c_j$ can be rewritten as 
\[\begin{cases}
c_1(y,\mu;w)=(w+\beta_1\mu)(w+\beta_2\mu)\\
c_2(y,\mu;w)=\left(2 \mathbbm{1}_{\{y>0\}}-1\right) 
\left(\beta_1w-\left(2 \mathbbm{1}_{\{y>0\}}-1\right) \beta_1\mu \right)\left(w-\left(2 \mathbbm{1}_{\{y>z\}}-1\right) \beta_2\mu\right)\\
c_3(y,\mu;w)=\left(2 \mathbbm{1}_{\{y>z\}}-1\right) 
\left(\beta_2 w-\left(2 \mathbbm{1}_{\{y>z\}}-1\right) \beta_2\mu \right)\left(w+\left(2 \mathbbm{1}_{\{y>0\}}-1\right) \beta_1\mu\right)\\
c_4(y,\mu;w)=\left(1-2 \mathbbm{1}_{\{0\leq y <z\}}\right) 
\left(\beta_1\beta_2 w^2-\beta_1 \beta_2\mu^2\right).
\end{cases}\]

Let $\frak{a}$ be a fixed real number and define 
\begin{equation} \label{eq:cja}
C_{j,0}:=c_{j,0}, \quad 
C_{j,1}:= \mu c_{j,1} + 2 c_{j,0} \frak{a} , \quad 
C_{j,2}:=c_{j,2} \mu^2 + c_{j,1} \mu \frak{a} + c_{j,0} \frak{a}^2.
\end{equation}

Let us then define some functions which will be used throughout the section. For any $K,m,n\in\mathbb{N}$, $\omega\geq 0$, and $\frak{a}, \tau \in\mathbb{R}$:
\begin{equation} \label{G}
 \mathscr{G}_{K,m,n}(\omega,\frak{a},\tau):=(-1)^{K} (K+m)! \sum_{\ell=0}^{\lfloor {\frac{K+m}{2}}\rfloor} \frac{(-1)^\ell }{2^\ell}\frac{1}{\ell! (K+m-2\ell)!} 
\mathcal{S}_{K+m-2 \ell, n}(\omega,\frak{a},\tau)
\end{equation}
where
\[\begin{split}
 \mathcal{S}_{L,n}(\omega,\frak{a},\tau) & =  \sum_{n'=0}^n \sum_{L'=0}^{L} {n \choose n'} {L \choose L'} (\omega+\tau)^{n-n'} (\frak{a}+\tau)^{L-L'} \mathcal{J}_{n'+L'}(\omega,\tau),
\end{split}\]
and $\mathcal{J}_{q}(\omega,\tau):= \disp e^{-\frac12 \omega^2} e^{\frac12 (\omega+\tau)^2} \int_{-\infty}^{\omega+\tau} w^q e^{-\frac12 w^2} d w$. 
The latter function satisfies the recursive relationship 
$$
\mathcal{J}_q(\omega,\tau)=(q-1) \mathcal{J}_{q-2}(\omega,\tau)+(-1)^{q-1}(\omega+\tau)^{q-1} \mathcal{J}_1(\omega,\tau),
$$ 
hence, following Lemma~2.17 in \cite{DMR},
\[
\mathcal{J}_q (\omega,\tau) := 
\begin{cases}
\sqrt{2 \pi} e^{-\frac12 \omega^2} e^{\frac12 (\omega+\tau)^2} \Phi^c(\omega+\tau) 
& q=0,\\
-   e^{-\frac12 \omega^2} & q=1, \\
\mathcal{J}_0(\omega,\tau) (q-1)!! - \mathcal{J}_1(\omega,\tau) \sum_{k=0}^{\frac{q}{2}-1} (\omega+\tau)^{q-2k-1} \frac{(q-1)!!}{(q-2k-1)!!}  & q\geq 2 \text{ even, }\\
\mathcal{J}_1(\omega,\tau) \sum_{k=0}^{\frac{q-1}{2}} (\omega+\tau)^{(q-1-2k)} 2^k \frac{(\frac{q-1}{2})!}{(\frac{q-1}{2}-k)!} & q \geq 3 \text{ odd },
\end{cases}
\]
where $(2n+1)!!=(2n+1)\cdot(2n-1)\cdot\ldots\cdot 3 \cdot 1$, $n \in \mathbb{N}$.

The following theorem is a non trivial generalization of Theorem 2.13 in \cite{DMR}, which proposed an explicit representation for the transition probability density $p^{(\beta_1,\beta_2)}_\mu(t,x,y)$ of $\mathbb{P}^{(\beta_1,\beta_2)}_\mu$
in the special case where $\beta_1\mu>0,\beta_2\mu>0$. The result proved here holds for any  parameters  $\beta_1,\beta_2$
and $\mu$.

\begin{thrm}\label{th:tdf2skewdrift}
Let $\beta_1,\beta_2\in (-1,1)$, $\mu\in\mathbb{R}$ and $\frak{a}\geq \max{\left(0,-2\beta_1\mu,-2\beta_2\mu\right)}$.
The transition density of the $(\beta_1,\beta_2)$-SBM with constant drift $\mu$ decomposes as
\[ 
p^{(\beta_1,\beta_2)}_\mu(t,x,y)=  p_\mu(t,x,y)  v^{(\beta_1,\beta_2)}_\mu(t,x,y)
\]
where the function $v_\mu^{(\beta_1,\beta_2)}$, which does not depend on $\frak{a}$, admits the following series representation for any  $\frak{a}$:
\begin{equation} \label{d2skewmu}
v^{(\beta_1,\beta_2)}_\mu(t,x,y) = \sum_{k=0}^{\infty}  (-\beta_1 \beta_2\mu^2 t)^k  \sum_{j=1}^4  F_{j,k}(\omega_{j,k},\frak{a}),
\end{equation}
where
\begin{equation}\label{eq:Fjkbeta}
F_{j,k}(\omega_{j,k},\frak{a}):= 
\begin{cases}
\disp \sum_{n=0}^{k}  {2 k -n \choose k} \frac{\frak{C}_{j,k}(\frak{a})}{n! (\beta_1\mu\sqrt{t}-\beta_2\mu\sqrt{t})^{2 k +1-n} } \mathscr{F}_{k+h { -s},m,n}(\omega_{j,k},\frak{a}), & \text{ if } \beta_1\neq \beta_2; \\
\disp (-1)^{k+1} \frac{\frak{C}_{j,k}(\frak{a}) }{(2 k+1)!} \mathscr{G}_{k+h {-s},m,2k+1}(\omega_{j,k},\frak{a}\sqrt{t},\beta_1\mu\sqrt{t}), & \text{ if }\beta_1=\beta_2;
\end{cases}
\end{equation}
\begin{equation} \label{omegajk}
\omega_{j,k}(x,y):=\frac{a_j(x,y)+2 z k +|y-x|}{\sqrt t}, \quad j=1,2,3,4, \quad k\in\mathbb{N};
\end{equation}
\[
\frak{C}_{j,k}(\frak{a}):= e^{\frac12 \omega_{1,0}^2} \sum_{m=0}^{k}  \sum_{s=0}^{k-m} \sum_{h=0}^2  C_{j,2-h}(y)  {k-m \choose s} {k \choose m}  \frac{ (-2 \frak{a}\sqrt{t})^{k-m{ -s}}  {(\mu^2-\frak{a}^2)^{s}}}{\mu^{2 k} t^{k {-s}} }
\]
with $C_{j,h}$ given in \eqref{eq:cja}, $a_j(x,y)$ in \eqref{aj} and 
the function $\mathscr{G}_{K,m,n}$ given by \eqref{G}.\\
The function $\mathscr{F}_{K,m,n}$ is defined by 
\[
\mathscr{F}_{K,m,n}(\omega,\frak{a}):=\mathscr{G}_{K,m,n}(\omega,\frak{a}\sqrt{t},\beta_2\mu\sqrt{t})-(-1)^n \mathscr{G}_{K,m,n}(\omega,\frak{a}\sqrt{t},\beta_1\mu\sqrt{t}).
\]
\end{thrm}

\begin{proof} First we need to recall some results presented in \cite{DMR}. The transition density of the $(\beta_1,\beta_2)$-SBM with constant drift $\mu$ has the integral representation 
$ \disp{
p^{(\beta_1,\beta_2)}_\mu (t,x,y)=\frac{1}{2\pi i}\int_{\Gamma} e^{\lambda t} G(x,y;\lambda) d\lambda,}$
where $\Gamma$ is a contour of $(-\infty,0]$ (where the possible singularities are located, since it is the complementary of the resolvent set) and $G(x,y;\lambda)$ was computed in Lemma 2.14 in \cite{DMR}. One can make the change of variable $\phi(\lambda)=\sqrt{2 \lambda +\mu^2}$ for $\lambda\in \mathbb{C}\setminus(-\infty,0]$ proceeding as in Figure~ \ref{fig:contourLine}.a and obtain as expression for $v_\mu^{(\beta_1,\beta_2)}(t,x,y)$:
\begin{equation}\label{eq:initialintegral} \begin{split}
\frac{p^{(\beta_1,\beta_2)}_\mu (t,x,y)}{p_\mu (t,x,y)} & = 
\frac{\sqrt{t}}{\sqrt{2\pi} i} e^{\frac{(y-x)^2}{2 t}}\int_{\phi(\Gamma)} e^{\frac{w^2}{2} t}  
  e^{- w |x-y|} \frac{ \sum_{j=1}^4 c_j(\mu,y;w)e^{-w a_j(x,y)}}{\beta_1 \beta_2 e^{-2 w z} (w^2-\mu^2)  +   (w+\beta_1 \mu)(w+\beta_2 \mu)}
   d w.
\end{split}\end{equation}

\begin{figure}
\begin{center}
\begin{tikzpicture}

\path(0,-3.5) node{a)};
\draw[-stealth] (-1,0) -- (3,0) node[above left]{$\mathbb{R}$};
\draw[-stealth] (0,-3) -- (0,3) node[above]{$i \mathbb{R}$};

\draw[red] (-1,0) -- (0,0);
\draw[red,dashed,thick] (-2.5,0) -- (-1,0) node[below]{$-\frac{\mu^2}2$};
\fill[red] (0,0) circle(0.05);
\draw (-1.05,-0.1) -- (-1,-0.1) -- (-1,0.1) -- (-1.05,0.1);

\draw[blue] (-2.5,-1) -- (0,-1) arc(-90:90:1)  -- (-2.5,1);
\draw[blue,->] (-2,-1) -- (-1,-1) ;
\draw[blue,->] (-1.5,1) node[above]{$\Gamma$} -- (-2,1);

\draw[rltgreen] (0.5,-3) ..controls(0.6,-2.5) and (0.7,-2.1).. (1.4,-1.4) arc(-45:45:2) ..controls(0.7,2.1) and (0.6,2.5).. node[above right]{$\phi(\Gamma)$} (0.5,3);
\draw[rltgreen,->] (2,0) arc(0:5:2);

\end{tikzpicture}
\hspace{1.5cm}
\begin{tikzpicture}[scale=1]

\path(0,-3.5) node{b)};
\draw[-stealth] (-1,0) -- (4,0);
\draw[-stealth] (0,-3) -- (0,3);

\draw[rltgreen,dashed,thick] (0.5,-3) ..controls(0.6,-2.5) and (0.7,-2.1).. node[below right]{$\phi(\Gamma)$} (1.4,-1.4) arc(-45:45:2) ..controls(0.7,2.1) and (0.6,2.5).. (0.5,3);

\draw[blue,thick,->] (2.9,-3)--(2.9,3) node[above]{$\frak{a}+i \mathbb{R}$};

\draw[orange] (2.9,2.25)node[right]{$\frak a + U$} -- node[below]{$\rho_U$} (0.73,2.25);
\fill[orange] (2.9,2.25) circle (0.05); \fill[orange] (0.73,2.25) circle (0.05);
\draw[orange] (2.9,-2.25) -- (0.73,-2.25);
\fill[orange] (2.9,-2.25)node[right]{$\frak a - U$} circle (0.05); \fill[orange] (0.75,-2.25) circle (0.05);

\draw[dockerblue,dashed,->] (1.6,1.2)--(2.8,1.2);
\draw[dockerblue,dashed,->] (1.9,0.6) -- (2.8,0.6);
\draw[dockerblue,dashed,->] (1.8,-1) -- (2.8,-1);
\draw[dockerblue,dashed,->] (1.2,-1.6) -- (2.8,-1.6);

\fill[red] (1.7,0) circle(0.06) node[below]{$|\mu|$};
\draw[red] (0,0) -- (1.7,0);
\draw[red] (0,0) circle(0.06);

\fill[red] (0.8,-0.1) node[below]{pole} rectangle (1,0.1) ;

\end{tikzpicture}
\end{center}
\vspace{0.5cm}
\captionsetup{singlelinecheck=off}
\caption[contourdrift]{\begin{itemize}
\item[a)] The picture shows the image of the contour $\Gamma$ under $\phi:\lambda \mapsto \sqrt{2 \lambda+\mu^2}$. The line $(-\infty,0]$ contains the spectrum of the operator $(L,\mathcal{D}(L))$. The dashed line $(-\infty,-\frac{\mu^2}2]$ is the complement of the domain of $\phi$.
\item[b)] The figure represents the vertical line $\frak{a}+i\mathbb{R}$ on the right of the curve $\phi(\Gamma)$ and the segment $\rho_U$ connecting the point $\frak a +U=\frak{a}+i u$  to its unique projection on $\phi(\Gamma)$.
The segment $(0,|\mu|]$ is the image under $\phi$ of $(-\frac{\mu^2}{2},0]$. Some singularities could exist in $[0,|\mu|]$.
\end{itemize}}\label{fig:contourLine}
\end{figure}

Let us take a non negative real number $\frak{a}$ as in Figure~\ref{fig:contourLine}.b. One can deform the contour $\phi(\Gamma)$ to the line $\frak{a}+i \mathbb{R}$ in Figure~\ref{fig:contourLine}.b. because the integrand is holomorphic on the region between the two curves, is continuous on the curves and on the segment $\rho_U$, and the integral on this segment is vanishing if $|U| \to \infty$ (implied by Lemma~\ref{lmm:rho}). In fact the possible singularities with positive real part could only be located in $[0,|\mu|]$, since it is the image through $\phi$ of $(-\infty,-{\mu^2/2}]$, see Figure~\ref{fig:contourLine}.a.

Noticing that $\omega_{j,0}=\frac{a_j(x,y)+|x-y|}{\sqrt{t}}$, one has
\[ \begin{split}
v^{(\beta_1,\beta_2)}_\mu & (t,x,y) = 
\frac{\sqrt{t}}{\sqrt{2\pi} i} e^{\frac12 \omega_{1,0}^2} \int_{\frak{a}+i \mathbb{R}} e^{\frac{w^2}{2} t}  
  e^{- w |x-y|} \frac{ \sum_{j=1}^4 c_j(\mu,y;w)e^{-w a_j(x,y)}}{\beta_1 \beta_2 e^{-2 w z} (w^2-\mu^2)  +   (w+\beta_1 \mu)(w+\beta_2 \mu)} d w \\
 & \overset{u=-i \sqrt{t}(w - \frak{a})}{=} \frac{1}{\sqrt{2\pi} } e^{\frac{(y-x)^2}{2 t}} \int_{\mathbb{R}} 
  \frac{ e^{\frac{(i u +\frak{a} \sqrt{t})^2}{2}} \sum_{j=1}^4 c_j(\mu\sqrt{t},y; i u +\frak{a} \sqrt{t})e^{-(i u +\frak{a} \sqrt{t}) \frac{a_j(x,y) +|x-y|}{\sqrt{t}}}}
 {\beta_1 \beta_2 e^{-2 (i u +\frak{a} \sqrt{t}) \frac{z}{\sqrt{t}}} ((i u+\frak{a} \sqrt{t})^2-\mu^2 t)  +   ((i u +\frak{a} \sqrt{t})+\beta_1 \mu\sqrt{t})((i u +\frak{a} \sqrt{t})+\beta_2 \mu\sqrt{t})} d u. 
\end{split}\]

For simplicity (but without loss of generality) we continue the computation supposing $t=1$.
Let us define $\frak{a}_\frak{i}=\frak{a}+\beta_\frak{i} \mu$, $\frak{i}=1,2$.

\begin{equation} \label{eq:integralbeta}
 \begin{split}
v^{(\beta_1,\beta_2)}_\mu & (1,x,y) 
=\frac{1}{\sqrt{2\pi} } e^{\frac{1}{2}(\omega_{1,0}^2+\frak{a}^2)} \int_{\mathbb{R}} \frac{ e^{-\frac{ u^2}{2}}  \sum_{j=1}^4 \sum_{h=0}^2 \left(- c_{j,2-h}(y) \mu^{2-h} (i u +\frak{a})^h \right)  e^{-i u (\omega_{j,0}-\frak{a})} e^{-\frak{a} \omega_{j,0}}}  { (u - i \frak{a}_1)( u - i \frak{a}_2) \left( 1+ e^{-i 2 z  u } e^{- 2 z \frak{a} }  \frac{\beta_1 \beta_2 (u^2-\frak{a}^2+\mu^2-2 i u \frak{a})}{(u - i \frak{a}_1)( u - i \frak{a}_2)}\right) } d u.
\end{split}
\end{equation}

Since we have assumed $\frak{a}\geq 0, \frak{a}\geq  -2\beta_1\mu, \frak{a}\geq -2\beta_2\mu$, it is easy to prove that $\left|- e^{- 2 z \frak{a}} \beta_1\beta_2 \frac{(v^2-\frak{a}^2+\mu^2-2 i v \frak{a})}{(v - i \frak{a}_1)( v - i \frak{a}_2)}\right|<1$. 
Therefore one factor of the integrand is a geometric series. Hence, exchanging integral and series one obtains the series of Fourier transforms
\begin{equation}\label{eq:Fourierseries}
\begin{cases}
v^{(\beta_1,\beta_2)}_\mu(1,x,y) = \disp \sum_{k=0}^{\infty}  (-\beta_1 \beta_2\mu^2)^k  \sum_{j=1}^4   F_{j,k}(\omega_{j,k},\frak{a}),\\
F_{j,k}(\omega_{j,k},\frak a):=e^{\frac{1}{2}(\omega_{1,0}^2+\frak{a}^2)} e^{-\frak{a} \omega_{j,k}} \mathcal{F}\left(w \mapsto  e^{-\frac{ w^2}{2}}  c_j(y,\mu;i w+\frak{a}) \frac{(w^2-\frak{a}^2+\mu^2-2 i w \frak{a})^k}{\mu^{2 k}} \cdot\frac{-1}{(w-i\frak{a}_1 )^{k+1}(w-i\frak{a}_2)^{k+1}} \right)(\omega_{j,k}-\frak a) .
\end{cases}
\end{equation}

We could exchange the integral and the series since, following the proof of Proposition \ref{boundtheta}, we can find a bound for the absolute value of the $k-th$ term of the series, such that the series of these bounds converges.

The Fourier transform in $F_{j,k}$ can be rewritten as the following convolution of Fourier transforms
\begin{equation}\label{eq:splitFourier}
\frac{1}{\sqrt{2\pi} \mu^{2 k}}\underbrace{\mathcal{F}\left(e^{-\frac{ w^2}{2}}   \left(\sum_{h=0}^2  c_{j,2-h}(y) \mu^{2-h} (i w +\frak{a})^h \right) (w^2-\frak{a}^2+\mu^2-2 i w \frak{a})^k\right)}_{(1)} * \underbrace{\mathcal{F}\left(\frac{-1}{\left[(w-i\frak{a}_1 )(w-i\frak{a}_2)\right]^{k+1}}\right)}_{(2)}.
\end{equation}

\begin{lmm}[Study of the term { (2)}] \label{lmm:Fourierf}
If $a >0$ and $k\in \mathbb{N}$, then
\[ \mathcal{F}\left(w \mapsto \frac{1}{(w- i a)^{k+1}}\right)(\omega)
= i^{k+1} \sqrt{2 \pi}  \, \frac{(-\omega)^k}{k!} \, g(\omega,a),\]
where
\begin{equation} \label{f:goma}
g(\omega,a)=e^{a \omega} \mathbbm{1}_{(-\infty,0)}(\omega).
\end{equation}
Let $a_1,a_2$ positive real numbers, then
\begin{equation} \label{f:Fourierf}
\begin{split}
 \mathcal{F} \left(w\mapsto\frac{-1}{(w- i a_1 )^{k+1}(w-i a_2)^{k+1}}\right)(\omega)& \\ 
= & \frac{\sqrt{2 \pi}}{(a_1-a_2)^{2 k +1}  k!} \cdot \sum_{n=0}^{k} \frac{\left(2k-n\right)!}{n! (k-n)!} (a_1-a_2)^n \omega^{n} \left[g(\omega,a_2)- (-1)^n g(\omega,a_1)\right]\\
= & :f_{k+1}(\omega,a_1,a_2) .
\end{split}
\end{equation}
\end{lmm}
\begin{proof}
See Lemma 2.15 and 2.16 in \cite{DMR}.
\end{proof}
Let us now consider the term {(1)} in \eqref{eq:splitFourier}.\\
The coefficients defined in (\ref{eq:cja}) are such that $\disp{ \sum_{h=0}^2  c_{j,2-h}(y) \mu^{2-h} (i w +\frak{a})^h = \sum_{h=0}^2  C_{j,2-h}(y) i^h w^h}$, hence the first term of the Fourier transform becomes
\[\mathcal{F}\left(e^{-\frac{ w^2}{2}}   \left(\sum_{h=0}^2  C_{j,2-h}(y) i^h w^h \right) (w^2-\frak{a}^2+\mu^2-2 i w \frak{a})^k\right).
\]
Developing the power of binomials one obtains
\[ \sum_{h=0}^2  C_{j,2-h}(y) \sum_{m=0}^k {k \choose m} \sum_{s=0}^{k-m} {k-m \choose s} (-2 \frak{a})^{k-m-s} (\mu^2-\frak{a}^2)^{s} \ i^{k-m-s+h} \mathcal{F}\left(e^{-\frac{ w^2}{2}}w^{k+m-s+h}\right).\]
Finally one computes the Fourier transforms $\disp{ \mathcal{F}\left(  e^{-\frac{w^2}{2}}  w^{n}\right) (\omega) =  i^{n} \frac{d^{n}}{d \omega^{n}} e^{-\frac{\omega^2}{2}}}$
and concludes that
\[\begin{split}
& \frac1{\mu^{2 k}} \mathcal{F}\left(e^{-\frac{ w^2}{2}}   \left(\sum_{h=0}^2  C_{j,2-h}(y) i^h w^h \right) (w^2-\frak{a}^2+\mu^2-2 i w \frak{a})^k\right)=\\
& \qquad { {=} \sum_{m=0}^k {k \choose m} \sum_{s=0}^{k-m} {k-m \choose s} (-2 \frak{a})^{k-m-s}  \frac{(\mu^2-\frak{a}^2)^{s}}{\mu^{ 2 k}} \ \sum_{h=0}^2  C_{j,2-h}(y) \ (-1)^{k+h-s} \frac{d^{k+m+h-s}}{d w^{k+m+h-s}} e^{-\frac{w^2}{2}}}.
\end{split}\]
Define now
\[\mathscr{G}_{K,m,n}(\omega,\frak{a},\beta_\frak i\mu):= e^{\frac{1}{2}\frak{a}^2} e^{-\frak{a} \omega} (-1)^{K} \left(\frac{d^{K+m}}{d w^{K+m}} e^{-\frac{w^2}{2}} * w^{n} g(w,{\frak{a}_\frak i}) \right)(\omega-\frak{a}) \quad \frak i=1,2.
\]
One has to show that this function coincides with the expression given in \eqref{G}.\\
Using that $\disp{\frac{d^{n}}{d w^{n}} e^{-\frac{w^2}{2}}=(-1)^n e^{-\frac{w^2}{2}} H_{n}(w),}$ where $H_n(w)$ are the Hermite polynomials, then
\[
\begin{split}
\mathscr{G}_{K,m,n} & (\omega,\frak{a},\beta_\frak i \mu) =e^{\frac{1}{2}\frak{a}^2} e^{-\frak{a} \omega}(-1)^m \left(w^{n}g(w,\frak{a}_\frak i)* H_{K+m}(w) e^{-\frac{w^2}{2}}\right)(\omega -\frak{a})\\
& = (-1)^m e^{-\frac{1}{2}\omega^2}  e^{\frac12 (\omega +\beta_\frak i \mu)^2}  \int_{-\infty}^{-(\omega +\beta_\frak i \mu)} (u+\omega +\beta_\frak i \mu)^{n} e^{-\frac{u^2}{2}}  H_{K+m}(-\frak a_\frak i-u) d u.
\end{split}
\]
One then uses the explicit expression of the Hermite polynomials
$  H_n(w)=n! \sum_{\ell=0}^{\lfloor {\frac{n}{2}}\rfloor} (-1)^\ell \frac{1}{2^\ell}\frac{1}{\ell! (n-2\ell)!} w^{n-2\ell} $.
 The recovered expression is the one in \eqref{G} once, the function defined by
\[\mathcal S_{L,n}(\omega,\frak{a},\beta_{\frak i}\mu) := e^{-\frac{1}{2}\omega^2}  e^{\frac12 (\omega+\beta_\frak i \mu)^2}  \int_{-\infty}^{-(\omega+\beta_\frak i \mu)} (u+\beta_\frak i \mu)^{n} e^{-\frac{u^2}{2}} (\frak a _\frak i+u)^{L} d u,\]
is expressed using 
$ \mathcal J_{q}(\omega,\beta_\frak{i} \mu):= e^{-\frac{1}{2}\omega^2}  e^{\frac12 (\omega+\beta_\frak i \mu)^2} \int_{-\infty}^{-(\omega+\beta_\frak{i}\mu)} u^{q} e^{-\frac{u^2}{2}} d u$.

\end{proof}

\begin{figure}
\begin{center}
\begin{tikzpicture}[scale=1]

\draw[-stealth] (-1.5,0) -- (3,0) node[above left]{$\mathbb{R}$};
\draw[purple,-stealth] (0,-3) -- (0,3) node[above]{$i \mathbb{R}$};
\path(0,-3.5) node{a)};

\draw[rltgreen] (0.5,-3) ..controls(0.6,-2.5) and (0.7,-2.1).. node[below right]{$\phi(\Gamma)$} (1.4,-1.4) arc(-45:45:2) ..controls(0.7,2.1) and (0.6,2.5).. (0.5,3);
\draw[orange] (0,2.25)  node[left]{$U$}-- node[below]{$\rho_U$} (0.73,2.25) node[right]{$U'$};
\draw[dockerblue,dashed,->] (1.4,1) -- (0.1,1);
\draw[dockerblue,dashed,->] (1.7,0.5) -- (0.1,0.5);
\draw[dockerblue,dashed,->] (1.7,-0.5) -- (0.1,-0.5);
\draw[dockerblue,dashed,->] (1.4,-1) -- (0.1,-1);
\fill[orange] (0,2.25) circle (0.05); \fill[orange] (0.73,2.25) circle (0.05);
\draw[orange] (0,-2.25)node[left]{-U} -- (0.73,-2.25);
\fill[orange] (0,-2.25) circle (0.05); \fill[orange] (0.73,-2.25) circle (0.05);

\draw[red] (1.3,0) circle(0.06) node[below right]{$|\mu|$};
\draw[red,dashed] (0,0) -- (1.3,0);
\draw[red] (0,0) circle(0.06);

\end{tikzpicture}
\hspace{0.5cm}
\begin{tikzpicture}

\path(0,-3.5) node{b)};
\draw[-stealth] (-1,0) -- (3,0) node[above left]{$\mathbb{R}$};
\draw[-stealth,rltgreen,thick] (0,-3) -- (0,3);

\draw[orange,thick] (0,2.25) node[left]{$U$} -- node[below]{$\rho_U$} (1.7,2.25) node[right]{$\frak{a}+U$};
\fill[orange] (0,2.25) circle (0.05);
\draw[orange, thick] (0,-2.25) node[left]{$-U$} -- (1.7,-2.25);
\fill[orange] (0,-2.25) circle (0.05);

\draw[blue,thick,->] (1.7,-3)--(1.7,3) node[above]{$\frak{a} +i \mathbb{R}$};
\draw[red,thick] (0,0) -- (1.3,0);
\draw[red] (0,0) circle(0.06);
\fill[red] (1.2,-0.1) node[below]{pole} rectangle (1.4,0.1) ;

\draw[rltred,dashed, thick] (0.05,2.20)--(1.65,2.20)--(1.65,-2.20)--(0.05,-2.20)--(0.05,2.20);

\end{tikzpicture} 
\hspace{0.5cm}
\begin{tikzpicture}

\path(0,-3.5) node{c)};
\draw[-stealth] (-1,0) -- (4,0) node[above left]{$\mathbb{R}$};
\draw[-stealth] (0,-3) -- (0,3);

\draw[orange,thick] (1.2,2) node[left]{$\frak{a}'+U$} -- node[below]{$\rho'_U$} (2.5,2) node[right]{$\frak{a}+U$};
\fill[orange] (1.2,2) circle (0.05);
\fill[orange] (2.5,2) circle (0.05);
\draw[orange, thick] (1.2,-2) node[left]{$\frak{a}'-U$} -- (2.5,-2);
\fill[orange] (1.2,-2) circle (0.05);
\fill[orange] (2.5,-2) circle (0.05);

\draw[purple,thick] (1.2,-2.5)--(1.2,2.5);
\draw[purple,thick,dashed] (1.2,-3)--(1.2,-2.5);
\draw[purple,thick,dashed,->] (1.2,2.5)--(1.2,3) node[above ] {$\frak{a}'+i \mathbb{R}$};
\draw[blue,dashed,->] (2.5,-3)--(2.5,3) node[above] {$\frak{a}+i \mathbb{R}$};

\draw[red,thick] (0,0) -- (0.6,0);
\draw[red] (0,0) circle(0.06);
\fill[red] (0.6,0) circle (0.1) node[above]{biggest} node[below]{pole};

\draw[dockerblue,dashed,->]  (2.5,1)--(1.2,1);
\draw[dockerblue,dashed,->]  (2.5,0.5)--(1.2,0.5);
\draw[dockerblue,dashed,->]  (2.5,-0.5)--(1.2,-0.5);
\draw[dockerblue,dashed,->]  (2.5,-1)--(1.2,-1);

\end{tikzpicture}
\end{center}
\vspace{0.5cm}
\captionsetup{singlelinecheck=off}
\caption[contourdrift]{\begin{itemize}
\item[a)] The picture shows the image $\phi(\Gamma)$ (in green) of the contour $\Gamma$ under $\phi$ 
which shrinks to the imaginary line (in blue). The segment $\rho_U$ connects the point $U' \in \phi(\Gamma)$ to its projection  $U=i u$.
\item[b)] The figure represents the imaginary line (in green), the line $\frak{a}+i \mathbb{R}$ (in blue) on the right of any pole and the rectangular cycle (in red dashed) around a pole situated in  $(0,|\mu|]$.
\item[c)] The segment $\rho'_U$ connects here the point $\frak{a}+U$ with its projection on $\frak{a}'+i\mathbb{R}$. The real number $\frak{a}'$ is chosen smaller than  $\frak{a}$ but larger than any pole.
\end{itemize}}\label{fig:commentsContour}
\end{figure}

Let us make some comments on how to optimize the choice of the parameter $\frak{a}$. \\
If there is no singularity located on the interval $(0,|\mu|]$, one can choose $\frak{a}=0$ and shrink the contour $\phi(\Gamma)$ to the imaginary line (see Figure~\ref{fig:commentsContour}.a). Thus one recovers the result obtained in \cite{DMR} in the simple case where the parameters of the dynamics   satisfy $\beta_1\mu>0,\beta_2\mu>0$.\\
If there are poles in $(0,|\mu|)$, the situation is much more delicate and the integral over any cycle around the poles leads to the residues. In Figure~\ref{fig:commentsContour}.b we illustrate this link with the residues method used in \cite{DMR}. Thanks to Lemma~\ref{lmm:rho} one can easily show that $\int_{\frak{a}+i \mathbb{R}}=\int_{i \mathbb{R}}+$ residues. 
One can then decide to integrate on a particular vertical line $\frak{a}'+i \mathbb{R}$, where $\frak{a}'$ is large enough to avoid the singularities of the integrand in (\ref{eq:initialintegral}), see Figure~\ref{fig:commentsContour}.c.\\
Therefore, once $\frak{a}$ is larger than any pole, the left hand side of \eqref{d2skewmu} does not depend on $\frak{a}$, although each term \eqref{eq:Fjkbeta} of the series do depend on $\frak{a}$. \\

Let us complete our proof showing that the integral over $\rho_U$ in Figure~\ref{fig:contourLine}.b (resp. over $\rho'_U$ in Figure~\ref{fig:commentsContour}.c.
) vanishes if $u\to +\infty$.
\begin{lmm}\label{lmm:rho}
Let $\frak{a}\geq 0$, and assume $|\beta_1\beta_2|\neq 1$. Then the integral 
of the function 
\[
w \mapsto e^{\frac{w^2}{2} t} e^{- w |x-y|} \frac{ \sum_{j=1}^4 c_j(\mu,y;w)e^{-w a_j(x,y)}}{\beta_1 \beta_2 e^{-2 w z} (w^2-\mu^2)  +   (w+\beta_1 \mu)(w+\beta_2 \mu)}
\]
on the segment $\rho_U=[iu;\frak{a}+iu]$ 
vanishes if $|u|\to \infty$.
\end{lmm}
\begin{proof}
One should prove that $\lim_{|u|\to \infty} I_u=0$ where
\[
I_u:= \int_0^{\frak{a}} e^{\frac{(\frak{w}+i u)^2}{2} t} e^{- (\frak{w}+i u) |x-y|} \frac{ \sum_{j=1}^4 c_j(\mu,y;\frak{w}+i u)e^{-(\frak{w}+i u) a_j(x,y)}}{\beta_1 \beta_2 e^{-2 (\frak{w}+i u) z} ((\frak{w}+ i u)^2-\mu^2)  +   (i u +\frak w+\beta_1 \mu)(i u +\frak w+\beta_2 \mu)} d \frak w.
\]
First
\[
|I_u|\leq e^{-\frac12 u^2} \int_0^{\frak a} e^{\frac{\frak{w}^2}{2} t} e^{- \frak{w} (|x-y|+a_j(x,y))} \frac{ \sum_{j=1}^4 | c_j(\mu,y;\frak{w}+i u) | }{|\beta_1 \beta_2 e^{-2 (\frak{w}+i u) z} ((\frak{w}+ i u)^2-\mu^2)  +   (i u +\frak w+\beta_1 \mu)(i u +\frak w+\beta_2 \mu)|} d \frak w.
\]
Let us consider separately the numerator and the denominator appearing in the integrand. It is straightforward to prove that the numerator is smaller than $K_\frak{a}\sum_{j=1}^4\sum_{h=0}^2  |\mu^{2-h}|\left(\frak{a}^2+u^2\right)^{\frac{h}{2}}$, where $K_\frak{a}$ is a positive constant.\\
The denominator is larger than
\[\begin{split}
& \sqrt{(u^2+(\frak{w}+\beta_1\mu)^2)(u^2+(\frak{w}+\beta_2\mu)^2)}-|\beta_1\beta_2| (u^2+\frak{w}^2+\mu^2) \geq (1-|\beta_1\beta_2|) u^2 -|\beta_1\beta_2|(\mu^2+\frak{a}^2),
\end{split}\]
which does not depend on $\frak{w}$ and is strictly positive for $u$ large enough.\\
Therefore, since $\int_0^\frak{a} e^{\frac{\frak{w}^2}2 t} e^{- \frak{w}(|x-y|+a_{j,k}(x,y))} d\frak{w} \leq \frak{a} e^{\frac12\frak{a}^2}$, one has 
\[\lim_{|u|\to\infty} |I_u| \leq \frak{a} e^{\frac12\frak{a}^2} \lim_{|u|\to\infty} e^{-\frac12 u^2} \frac{ K_\frak{a} \sum_{j=1}^4\sum_{h=0}^2|\mu^{2-h}|\left(\frak{a}^2+u^2\right)^{\frac{h}{2}}}{(1-|\beta_1\beta_2|) u^2 -|\beta_1\beta_2|(\mu^2+\frak{a}^2)}=0\]
\end{proof}


In fact we are mainly interested in the asymptotic regime for small skewness: $\k$ large integer, $\beta_1=\frac{1}{\k}$, $\beta_2 (\k) = \frac{\theta_2} {\k \theta_1 +  b(0)-b(z)}$ and 
$ \mu (\k)=b(0)+ \k \theta_1$ as in Lemma \ref{weakconv}. Recall
 that $\lim_{\k \to\infty} \beta_2 (\k) = 0$ and  $\lim_{\k \to\infty} \beta_j (\k) \mu (\k) = \theta_j, j=1,2$.\\
Therefore, we define the polynomials $\tilde c_j(y;w):=\lim_{\k \to \infty} c_j(y,\mu(\k)\sqrt{t};w)$, $j=1,2,3,4$. They satisfy 
\begin{equation}\label{ctheta}
\begin{cases}
\tilde c_1(y;w)=(w+\theta_1\sqrt{t})(w+\theta_2\sqrt{t})\\
\tilde c_2(y;w)=- \theta_1 \sqrt{t} \left(w-\left(2 \mathbbm{1}_{\{y>z\}}-1\right) \theta_2 \sqrt{t}\right)\\
\tilde c_3(y;w)=-  \theta_2 \sqrt{t} \left(w+\left(2 \mathbbm{1}_{\{y>0\}}-1\right) \theta_1 \sqrt{t}\right)\\
\tilde c_4(y;w)=-\left(1-2 \mathbbm{1}_{\{0\leq y <z\}}\right) 
\theta_1 \theta_2 t.
\end{cases}
\end{equation}
Their coefficients $\tilde c_{j,h}$, defined by the relationship $\tilde c_j(y;w)=\sum_{h=0}^2 \tilde c_{j,2-h}(y) w^h$, satisfy
\[\begin{cases}
\tilde c_{1,0}(y)=1,	\\
\tilde c_{2,0}(y)=0 	\\
\tilde c_{3,0}(y)=0 	\\
\tilde c_{4,0}(y)=0 
\end{cases},
\begin{cases}
\tilde c_{1,1}(y)=\theta_1 \sqrt{t}+\theta_2 \sqrt{t}	\\
\tilde c_{2,1}(y)=-\theta_1	\sqrt{t}		\\
\tilde c_{3,1}(y)=-\theta_2 \sqrt{t}			\\
\tilde c_{4,1}(y)=0
\end{cases},
\begin{cases}
\tilde c_{1,2}(y)=\theta_1\theta_2 t \\
\tilde c_{2,2}(y)=\left(2\mathbbm{1}_{\{y\geq z\}}-1\right)\theta_1\theta_2 t \\
\tilde c_{3,2}(y)=-\left(2\mathbbm{1}_{\{y>0\}}-1\right)\theta_1 \theta_2 t \\
\tilde c_{4,2}(y)=-\left(1-2\mathbbm{1}_{[0,z)}(y)\right)\theta_1\theta_2 t.
\end{cases}\]
They are obtained as the limit for $\k\to\infty$ of $ \mu(\k)^h \, c_{j,h}$, with $c_{j,h}$ given  by \eqref{eq:cj}.
Finally, for any $\frak{a}$, $\tilde C_{j,h}$ are defined, analogously to (\ref{eq:cja}), by
\begin{equation} \label{eq:tildecja}
\tilde C_{j,0}=\tilde c_{j,0}, \quad 
\tilde C_{j,1}= \tilde c_{j,1} + 2  \tilde c_{j,0} \frak{a}, \quad 
\tilde C_{j,2}=\tilde c_{j,2} +\tilde c_{j,1} \frak{a} + \tilde c_{j,0} \frak{a}^2,
\quad j=1,2,3,4.
\end{equation}

\begin{prpstn} \label{vthetalimit}
Let $\theta_1,\theta_2 \in \mathbb{R}$.
Let us denote by $v^{(\theta_1,\theta_2)}$ the pointwise limit for $\k\to \infty$ of the functions $v^{(1/\k,\beta_2(\k))}_{\mu(\k)}$ defined by \eqref{d2skewmu}. Recall that $\omega_{j,k}$ is defined by $(\ref{omegajk})$ and $\tilde C_{j,h}$ by $\eqref{eq:tildecja}$.
Let $\frak{a}\geq 0$ such that $\frak{a}>\max{ \left(-2\theta_1,-2\theta_2\right)}$.
Then the following representation holds
{
\begin{equation}\label{vtheta}
v^{(\theta_1,\theta_2)}(t,x,y)= \sum_{k=0}^\infty (-\theta_1\theta_2)^k t^k \sum_{j=1}^4 \tilde F_{j,k}(\omega_{j,k}(x,y),\frak{a}\sqrt{t})
\end{equation}
}
where $\tilde F_{j,k}$ satisfies
\begin{equation}\label{eq:FjkFourier}
 \tilde F_{j,k}(\omega_{j,k},\frak a)= e^{\frac1{2} \left(\omega_{1,0}^2+\frak{a}^2 \right)} e^{- \frak{a}\omega_{j,k}} \mathcal{F}\left(w \mapsto e^{-\frac{w^2}{2}} \tilde c_j(y;iw+\frak{a})  \frac{-1}{(w-i \tilde{\frak{a}}_1)^{k+1}(w-i \tilde{\frak{a}}_2)^{k+1}} \right)(\omega_{j,k}-\frak{a})
\end{equation}
where $\tilde{ \frak{a}}_\frak{i}:=\frak{a} \sqrt{t}+ \theta_\frak{i} \sqrt{t}$.
\end{prpstn}
\begin{proof} Without loss of generality we can prove the statement for $t=1$. To prove it, it is sufficient to pass to the limit into the integral \eqref{eq:integralbeta}.
One then finds (\ref{vtheta}) and (\ref{eq:FjkFourier}). Proposition \ref{boundtheta} guarantees that integral and series can be exchanged. \\
Remark that $\tilde F_{j,k}$ can also be defined as the pointwise limit of the functions $F_{j,k}$ given in \eqref{eq:Fjkbeta} and therefore
\begin{equation} \label{eq:seriestheta}
\tilde F_{j,k}(\omega_{j,k},\frak{a}) = 
\begin{cases}
\sum_{n=0}^k \left(\frac{  (2k-n)!}{(k-n)! n! k! }\frac{e^{\frac12 \omega_{1,0}^2}}{(\theta_1-\theta_2)^{2k-n+1}}  \sum_{h=0}^2  \tilde C_{j,2-h}(y) \tilde{\mathscr{F}}_{h,n}(\omega_{j,k},\frak{a})\right),& \theta_1\neq \theta_2;\\
\frac{(-1)^{k+1}}{(2 k+1)!} \sum_{h=0}^2  \tilde{C}_{j,2-h}(y) e^{\frac12 \omega_{1,0}^2}\mathscr{G}_{h,0,2k+1}(\omega_{j,k},\frak{a},\theta_1), & \theta_1=\theta_2,
\end{cases}
\end{equation}
where $\disp{\tilde{\mathscr{F}}_{h,n}(\omega,\frak{a}):=\mathscr{G}_{h,0,n}(\omega,\frak{a},\theta_2)-(-1)^n \mathscr{G}_{h,0,n}(\omega, \frak{a}, \theta_1)}$. The function $\mathscr{G}_{h,0,n}$ was defined by \eqref{G}.\\
Notice that, due to our appropriate choice of $\frak{a}$, we can obtain the latter formula from \eqref{eq:FjkFourier} proceeding as for Theorem~\ref{th:tdf2skewdrift}. Indeed, since $\frak{a}$ is strictly larger than any pole of the limit expression, 
 there exists $\k_0$ such that, for any $\k>\k_0$, $\frak{a}$ is larger than $\max(-2\frac1{\k}\mu(\k), -2\beta_2(\k)\mu(\k))$.
\end{proof}

\subsection{\ Towards a uniform bound for $v^{(\theta_1,\theta_2)}$}

The main result in this section is the following proposition.
\begin{prpstn}[Uniform bound for $(x,y)\mapsto v^{(\theta_1,\theta_2)}(t,x,y)$] \label{boundtheta}
Let $\theta_1,\theta_2$ be any real numbers.
There exists a positive constant $C$, depending only on $\theta_1,\theta_2$, such that 
\begin{equation}\label{eq:boundv}
\sup_{x,y}\left|v^{(\theta_1,\theta_2)}(t,x,y) \right|\leq  \ \frac{C}{1-e^{-\frac{2 z^2}{t}}} \, .
\end{equation}
More precisely, one can take
\begin{equation} \label{BoundCtheta}
C=\begin{cases} \disp
1+ \max\left\{\psi(\theta_1,\theta_2),\psi(\theta_2,\theta_1)\right\} + \min\left\{1, \left|\frac{\theta_1\theta_2 \sqrt{t}}{\theta_1-\theta_2}\right| \left|\varphi(\theta_1\sqrt{t})-\varphi(\theta_2\sqrt{t})\right|\right\}  & \text{ if } \theta_1\neq \theta_2,\\
1+2 \sqrt{t} |\theta_1|\varphi(\theta_1 \sqrt{t})+3 \theta_1^2 t  & \text{ if } \theta_1=\theta_2,
\end{cases}
\end{equation}
where 
\begin{equation} \label{phi}
\varphi(w):=\sqrt{2 \pi }e^{\frac{w^2}{2}} \Phi^c(w)
\end{equation}
and
\[\psi(\theta_1,\theta_2):= |\theta_ 1| \sqrt{t} \varphi(\theta_{1} \sqrt{t})+ |\theta_2| \sqrt{t} \varphi(\theta_{2} \sqrt{t})+ \min\left\{2, \left(\left|\frac{\theta_1+\theta_2}{\theta_1-\theta_2}\right|-1\right) |\theta_1| \sqrt{t}\varphi(\theta_1 \sqrt{t})+2 \left|\frac{\theta_1\theta_2 \sqrt{t}}{\theta_1-\theta_2}\right| \varphi(\theta_2 \sqrt{t}) \right\}.\]

Moreover, if $R_N v^{(\theta_1,\theta_2)}$ denotes the remainder after the $(N+1)-th$ term of the series \eqref{vtheta} which represents $v^{(\theta_1,\theta_2)}$, then 
\begin{equation}\label{eq:boundrest}
\sup_{x,y}\left| R_N v^{(\theta_1,\theta_2)} (t,x,y)\right|\leq \ \frac{C}{1-e^{-\frac{2z^2}{t}}} e^{-\frac{2z^2}{t} (N+1)}.
\end{equation}
\end{prpstn}
Considering formula (\ref{vtheta}), it is clear that the proof of (\ref{eq:boundv}) is complete as soon as one finds an appropriate bound for $\sup_{x,y} \tilde{F}_{j,k}(\omega_{j,k},\frak{a}\sqrt{t})$ for each $j\in \{1,2,3,4\}$ and $k\in \mathbb{N}$. This is done in the next lemma.

\begin{lmm} \label{termboundtheta}
Let define $\tilde{F}_{j,k}$ by  \eqref{eq:FjkFourier} and $\omega_{j,k}$ by \eqref{omegajk} for any $j=1,2,3,4$ and $k\in\mathbb{N}$. Then there exists a positive constant $C_j$, depending on $j$ and $(\theta_1,\theta_2)$ but not on $k$, such that
\[ 
\sup_{x,y} \left| \tilde{F}_{j,k}(\omega_{j,k},\frak{a}\sqrt{t}) \right| \leq C_j \ 
 \Big( \frac{e^{-\frac{2 z^2}{t}}}{|\theta_1\theta_2 t |} \Big)^k.
\] 
\end{lmm}

\begin{proof}[Proof of Lemma \ref{termboundtheta}] 
It becomes straightforward once one shows that, for each $j=1,\ldots,4$ and $k\in\mathbb{N}$, there exists a positive constant $C_j$ not depending on $k$ such that
\begin{equation} \label{ineq:323}
\sup_{x,y} \left| \tilde{F}_{j,k}(\omega_{j,k},\frak{a}\sqrt{t})\right|
\leq C_j \, \frac{e^{-\frac{(\omega_{j,k}^2-\omega_{1,0}^2)}{2}}}{|\theta_1\theta_2 t |^k}.
\end{equation} 
Indeed, since $a_j(x,y)\geq 0$ for all $j=1,2,3,4$, the following estimate holds
$$
\frac12(\omega_{1,0}^2 -\omega_{j,k}^2 ) 
\leq - \frac{1}{2t}\left(a_j(x,y)+2 k z\right)^2 \leq - \frac{2 z^2}{t} k .
$$

Let us prove \eqref{ineq:323}. To simplify the notation, in the rest of the proof, one take $t=1$.
Then we define $\tilde{\frak{a}}_{\frak i}=\frak{a}+\theta_\frak{i}$ for $\frak{i}=1,2$ and also  
\[
\disp \tilde{p}_j(w):=\frac{-\tilde c_j(y;i w+\frak{a})}{(w-i\tilde{\frak{a}}_1)(w-i \tilde{\frak{a}}_2)}
\]
 where the polynomials $\tilde c_j(y;w)$ are given by (\ref{ctheta}). 
Equations (\ref{vtheta}, \ref{eq:FjkFourier}) can be rewritten as
\[\begin{cases}
v^{(\theta_1,\theta_2)}(1,x,y) = \sum_{k=0}^{\infty} (-\theta_1\theta_2)^k \sum_{j=1}^4 \tilde{F}_{j,k}(\omega_{j,k},\frak{a}),\\
\tilde{F}_{j,k}(\omega_{j,k},\frak{a})=  
\disp - e^{\frac1{2} \left(\omega_{1,0}^2+\frak{a}^2 \right)} e^{- \frak{a}\omega_{j,k}} \,  \mathcal{F}\left(\tilde{p}_j(w) e^{-\frac{w^2}2}\right)*\frac{f_k(w,\tilde{\frak a}_1,\tilde{\frak a}_2)}{\sqrt{2\pi}} (\omega_{j,k}-\frak{a}), & \quad j=1,2,3,4,
\end{cases}\]
where $f_k$ is defined in (\ref{f:Fourierf}).

The rest of the proof of Lemma \ref{termboundtheta} is tricky and divided in several steps. \\
First one needs to compute the Fourier transform of the product between $\tilde{p}_j$ and a Gaussian kernel: it will lead to the product between a  Gaussian kernel and a linear combination of translations of the function $\varphi$ defined by \eqref{phi}. Then one has to compute the convolution of this quantity with the function $f_k(\cdot,\tilde{\frak a}_1,\tilde{\frak a}_2)$, which is positive with support on $(-\infty,0]$. Its $L^1$-norm is computed in Lemma~\ref{L1}. We will complete the proof finding an estimate of the convolution which uses this latter information.\\
Let us proceed in giving explicitly the functions $\tilde{p}_j $, $j=1,2,3,4$. \\
\[\begin{cases}
\tilde{p}_1(w)=1\\
\tilde{p}_2(w) 
 =\disp \theta_1 \frac{i}{w-i\tilde{\frak{a}}_1} + 2\mathbbm{1}_{\{y\geq z\}} \theta_1\theta_2 \frac{-1}{(w-i\tilde{\frak{a}}_1) (w-i \tilde{\frak{a}}_2)}\\
\tilde{p}_3 (w)
= \disp -\theta_2 \frac{-i}{w-i \tilde{\frak{a}}_2} - 2\mathbbm{1}_{\{y\leq 0\}} \theta_1\theta_2 \frac{-1}{(w-i \tilde{\frak{a}}_1)(w-i \tilde{\frak{a}}_2)}\\
\tilde{p}_4 (w) 
= \disp -\left(1-2\mathbbm{1}_{[0,z)}(y)\right) \theta_1\theta_2 \frac{-1}{(w-i \tilde{\frak{a}}_1)(w-i \tilde{\frak{a}}_2)}.
\end{cases}
\]
Thus, if $\theta_1\neq \theta_2$ then
\[\begin{cases}
	\mathcal{F}\left(\tilde{p}_1(w) e^{-\frac{w^2}{2}}\right)(\omega)=e^{-\frac{\omega^2}{2}}\\
	\mathcal{F}\left(\tilde{p}_2(w) e^{-\frac{w^2}{2}}\right)(\omega)= e^{-\frac{\omega^2}{2}} \left( - \theta_1 \frac{\theta_1+(2\mathbbm{1}_{\{y \geq z\}}-1)\theta_2}{\theta_1-\theta_2}\, \varphi(\omega+\tilde{\frak{a}}_1) +2\mathbbm{1}_{\{y\geq z\}} \frac{\theta_1\theta_2}{\theta_1-\theta_2} \, \varphi{(\omega+\tilde{\frak{a}}_2)}\right)  \\
	\mathcal{F}\left(\tilde{p}_3(w) e^{-\frac{w^2}{2}}\right)(\omega)= e^{-\frac{\omega^2}{2}}\left(- \theta_2 \frac{\theta_2 -\left(2\mathbbm{1}_{\{y>0\}}-1\right) \theta_1}{\theta_1-\theta_2} \, \varphi{(\omega+\tilde{\frak{a}}_2)} -2\mathbbm{1}_{\{y \leq 0\}} \frac{\theta_1\theta_2}{\theta_1-\theta_2} \, \varphi{(\omega+\tilde{\frak{a}}_1)}\right)\\
	\mathcal{F}\left(\tilde{p}_4(w) e^{-\frac{w^2}{2}}\right)(\omega)
=e^{-\frac{\omega^2}{2}}
\left(2\mathbbm{1}_{[0,z)}(y)-1\right)  \frac{\theta_1\theta_2}{\theta_1-\theta_2} \, \left( \varphi{(\omega+\tilde{\frak{a}}_2)} - \varphi{(\omega+\tilde{\frak{a}}_1)} \right).
\end{cases}
\]
The Gaussian component is positive and decreasing on $(0,+\infty)$, and the function $\varphi$, defined by \eqref{phi}, is positive and decreasing on $\mathbb{R}$.\\
The bound of the convolutions
$
\left|\mathcal{F}\left(\tilde{p}_j(w) e^{-\frac{w^2}{2}}\right) * \frac{f_k(w,\tilde{\frak a}_1,\tilde{\frak a}_2)}{\sqrt{2\pi}}\right|(\omega_{j,k}-\frak{a}) $
is then controlled by 
\[\left|\left(e^{-\frac{w^2}{2}} \varphi(\tilde{w+\frak{a}}_\frak i)\right) * {f_k(w,\tilde{\frak a}_1,\tilde{\frak a}_2)}\right|(\omega_{j,k}-\frak{a})= \left|\int_\mathbb{R} e^{-\frac{(\omega_{j,k}-\frak{a}-y)^2}{2}} \varphi{(\theta_\frak{i}+\omega_{j,k}-y)} f_k(y,\tilde{\frak a}_1,\tilde{\frak a}_2) dy \right|, \quad \frak i=1,2.\]
Notice that $f_k$ has support on ${(-\infty,0]}$.
Then
\[\begin{split} \left|\left(e^{-\frac{w^2}{2}} \varphi(w+\tilde{\frak{a}}_\frak{i})\right) * \frac{f_k(w,\tilde{\frak a}_1,\tilde{\frak a}_2)}{\sqrt{2\pi}}\right|(\omega_{j,k}-\frak{a}) 
& \leq \varphi(\omega_{j,k}+\theta_\frak{i}) e^{\frak{a} \omega_{j,k}} e^{-\frac{\frak{a}^2}{2}} \int_{-\infty}^{0} e^{-\frac{(\omega_{j,k}-y)^2}{2}} \left| e^{-\frak{a} y}\frac{f_k(y,\tilde{\frak{a}}_1,\tilde{\frak{a}}_2)}{\sqrt{2\pi}}\right| dy\\
& \leq  \varphi(\omega_{j,k}+\theta_\frak{i}) e^{-\frac{(\omega_{j,k}-\frak{a})^2}{2}} \frac1{\sqrt{2\pi}}\| e^{-\frak{a} y}f_k(y,\tilde{\frak{a}}_1,\tilde{\frak{a}}_2) \|_{L^1} .
\end{split}\]
Suppose we knew that $ \| e^{-\frak{a} \cdot}f_k(\cdot,\tilde{\frak{a}}_1,\tilde{\frak{a}}_2) \|_{L^1}=\frac{\sqrt{2\pi}}{|\theta_1\theta_2|^k}$,  as we will prove in Lemma~\ref{L1} below.
We can then complete the proof in the following way.
Since $\omega_{j,k}$ is non negative, $\varphi(\omega_{j,k}+\theta_\frak{i})\leq \varphi(\theta_\frak{i})$ and
\[
\begin{split}
|v^{(\theta_1,\theta_2)}(1,x,y)|& =\sum_{k=0}^{\infty} |\theta_1\theta_2|^k \sum_{j=1}^4 \left| \tilde{F}_{j,k}(\omega_{j,k},\frak{a})\right| 
\leq \sum_{k=0}^{\infty}  \sum_{j=1}^4 m_j \, e^{-\frac{\omega_{j,k}^2-\omega_{1,0}^2}{2}} 
\end{split}
\]
where the coefficients $m_j$ are given by
\[\begin{cases} m_1=1\\
m_2 = |\theta_1|\varphi(\theta_1)+ \mathbbm{1}_{\{y \geq z\}} \min\left(2 , \, 2 \left|\frac{\theta_1\theta_2}{\theta_1-\theta_2}\right|\varphi(\theta_2) +|\theta_1| \left(\left|\frac{\theta_1+\theta_2}{\theta_1-\theta_2}\right|-1\right) \varphi(\theta_1)\right)\\
m_3 = |\theta_2|\varphi(\theta_2)+ \mathbbm{1}_{\{y <0 \}} \min\left(2 , \, 2 \left|\frac{\theta_1\theta_2}{\theta_1-\theta_2}\right|\varphi(\theta_1) +|\theta_2| \left(\left|\frac{\theta_1+\theta_2}{\theta_1-\theta_2}\right|-1\right) \varphi(\theta_2)\right)\\
m_4= \min \left( 1, \, \left|\frac{\theta_1\theta_2}{\theta_1-\theta_2}\right|\left|\varphi(\theta_1)-\varphi(\theta_2)\right|\right).
\end{cases}\]
The fact that $e^{-\frac{\omega_{j,k}^2-\omega_{j,0}^2}{2}} \leq e^{-\frac{2 z^2}{t}k}$ yields the conclusion. 

If $\theta_1=\theta_2$ then, analogously,
\[\begin{cases}
	\mathcal{F}\left(\tilde p_1(w) e^{-\frac{w^2}2}\right)(\omega) =e^{-\frac{\omega^2}2}\\
	\mathcal{F}\left(\tilde p_2(w) e^{-\frac{w^2}2}\right) (\omega) =e^{-\frac{\omega^2}2} \left[- \theta_1 \varphi(\omega+\tilde{\frak{a}}_1)+2\mathbbm{1}_{\{y\geq z\}}\theta_1^2 \left(1- (\omega+\tilde{\frak{a}}_1) \varphi(\omega+\tilde{\frak{a}}_1)\right) \right]\\
	 \mathcal{F}\left(\tilde p_3(w) e^{-\frac{w^2}2}\right)(\omega)  =e^{-\frac{\omega^2}2} \left[- \theta_1 \varphi(\omega+\tilde{\frak{a}}_1)-2\mathbbm{1}_{\{y< 0\}}\theta_1^2 \left(1- (\omega+\tilde{\frak{a}}_1) \varphi(\omega+\tilde{\frak{a}}_1)\right) \right]\\
	 \mathcal{F}\left(\tilde p_4(w) e^{-\frac{w^2}2}\right)(\omega) 
=e^{-\frac{\omega^2}2} \left(2\mathbbm{1}_{[0,z)}(y)-1\right)\theta_1^2 \left(1- (\omega+\tilde{\frak{a}}_1) \varphi(\omega+\tilde{\frak{a}}_1)\right).
\end{cases}\]
Following the same procedure as before, the conclusion comes from the fact that $0\leq w \, \varphi(w) \leq 1$ for $w\geq 0$.

\end{proof}

\begin{lmm} \label{L1}
Suppose $a\neq 0$, $a_1,a_2 >0$ and $k\in \mathbb{N^*}$. Then 
\[
\left\|\mathcal{F}\left(w \mapsto \frac1{(w-i a)^{k}}\right)(\omega)\right\|_{L^1}=\frac{\sqrt{2\pi}}{|a|^k} \quad \textrm{ and } \quad 
\|f_k(\cdot,a_1,a_2)\|_{L^1}=  \frac{\sqrt{2\pi}}{(a_1 a_2 )^k},
\]
where $f_k$ is defined in \eqref{f:Fourierf}. \end{lmm}

\begin{proof} If $a>0$,  by Lemma \ref{lmm:Fourierf}, 
$
\left|\mathcal{F}\left(w \mapsto \frac1{(w-i a)^{k+1}}\right)(\omega)\right|=\sqrt{2\pi} \mathbbm{1}_{\mathbb{R}^-}(\omega) \frac{|\omega|^k}{k!} e^{a \omega}
$.\\
Integrating by part,
\[\begin{split}
& \left\|\mathcal{F}\left(w \mapsto \frac1{(w-i a)^{k+1}}\right)(\omega)\right\|_{L^1}=\frac{\sqrt{2\pi}}{ a\  k!}\int_{-\infty}^0 a\ (-\omega)^k e^{a \omega}\, d\omega =\frac{1}{a} \left\|\mathcal{F}\left(w \mapsto \frac1{(w-i a)^{k}}\right)(\omega)\right\|_{L^1}.
\end{split}\]
So the first identity follows from the inductive hypothesis.\\
 If $a<0$ then 
$
\mathcal{F}\left(w \mapsto \frac1{(w-i a)^{k+1}}\right)(\omega)=-i^{k+1}\sqrt{2\pi}\frac{(-\omega)^k}{k!} e^{a\omega}\mathbbm{1}_{\mathbb{R}^+}(\omega)
$
 and the proof works as well.\\
To compute the second norm proceed as follows:
\[\begin{split}
\|f_{k+1}(\cdot,a_1,a_2)\|_{L^1}& =\frac1{\sqrt{2\pi}}\int_\mathbb{R} \left|\mathcal{F}\left(\frac{1}{(w-i a_1)^{k+1}}\right)*\mathcal{F}\left(\frac{1}{(w-i a_2)^{k+1}}\right)(\omega)\right| d\omega =\\
& = \int_\mathbb{R}\left| (-1)^{k} \sqrt{2\pi} (k!)^2 \int_\mathbb{R} (-y)^k e^{a_1 y} \mathbbm{1}_{\mathbb{R}_-}(y) (y-\omega)^k e^{a_2 (\omega-y)} \mathbbm{1}_{\mathbb{R}_-}(\omega-y) dy \right| d\omega
\end{split}\]
Since the integrand is positive, we can exchange the integration order and then use the previous result to conclude:
\[\begin{split}
\|f_{k+1}(\cdot,a_1,a_2)\|_{L^1} & =  \sqrt{2\pi} (k!)^2 \left(\int_\mathbb{R}  (-y)^k e^{a_1 y} \mathbbm{1}_{\mathbb{R}_-}(y) \, dy\right) \left( \int_\mathbb{R} (-y)^k e^{a_2 y} \mathbbm{1}_{\mathbb{R}_-}(y)  dy \right) \\
& = \frac1{\sqrt{2\pi}} \|\mathcal{F}\left(\frac{1}{(w-i a_1)^{k+1}}\right)\|_{L^1} \|\mathcal{F}\left(\frac{1}{(w-i a_2)^{k+1}}\right)\|_{L^1}= \frac{ \sqrt{2\pi}}{a_1^{k+1} a_2^{k+1}}.
\end{split}\]
\end{proof}

\begin{rmrk}\label{termbound}
The same technique remains valid to obtain a uniform bound for the approximation functions $v^{(\beta_1,\beta_2)}_\mu$: there exists a positive constant $C^{(\beta)}$ (smaller than 3) such that
\[
\sup_{x,y}\left|v^{(\beta_1,\beta_2)}_\mu(t,x,y) \right|\leq \ \frac{C^{(\beta)} }{1-e^{-\frac{2 z^2}{t}}} \qquad 
\text{and} 
\qquad \sup_{x,y}\left| R_N v^{(\beta_1,\beta_2)}_\mu (t,x,y)\right|\leq \ \frac{C^{(\beta)} }{1-e^{-\frac{2z^2}{t}}} e^{-\frac{2z^2}{t} (N+1)}.
\]
\end{rmrk}

\subsection{Sampling under $\mathfrak{Q}$} \label{samplingQ}

According to Section \ref{limitalgorithm} one can produce the skeleton of random variates from expression (\ref{limitsklaw}), (i.e. sample under $\mathfrak{Q}$) through generalized rejection sampling schemes for both functions $h^{(\theta_1,\theta_2)}$ and $q^{(\theta_1,\theta_2)}$ defined in (\ref{qtheta}).
Thanks to Proposition \ref{boundtheta}, it is now possible to apply the ideas presented in \eqref{fractionrejection}.

Since $v^{(\theta_1,\theta_2)}$ is a bounded series (see \eqref{vtheta}  and \eqref{eq:boundv}), let us denote the renormalized series, its truncation at the $(N+1)-th$ term and the bound for the remainder respectively as
\begin{equation}\label{eq:renormalizedseries}
\bar v^{(\theta_1,\theta_2)}(t,x,y):= \frac{1-e^{-\frac{2 z^2}{t}}}{C} v^{(\theta_1,\theta_2)} (t,x,y),  \quad \bar v_N^{(\theta_1,\theta_2)}(t,x,y), \quad \frak{R}_N \bar v^{(\theta_1,\theta_2)}(t)=e^{-\frac{2 z^2}{t} (N+1)},\\
\end{equation}
where $z$ is the distance between the barriers and $C$ is given in \eqref{BoundCtheta}.
For any fixed $\delta\in(0,1)$, the density ${h^{(\theta_1,\theta_2)}}(y)$ satisfies
\begin{equation}\label{GRSh}\begin{split}
\frac{{h^{(\theta_1,\theta_2)}}(y)}{p_0(\frac{T}{1-\delta}, x_0,y)} & = \underbrace{ \frac{C_{\theta,x_0,T}}{\sqrt{1-\delta}} \ e^{M_B} \frac{C}{1-e^{-\frac{2 z^2}{T}}}}_{C^{\mathcal{H}}} \
\underbrace{ {\frac{e^{-\frac{(y-x_0)^2}{2 T} \delta} \  e^{B(y)-B(x_0)} }{e^{M_B }} \ \bar v^{(\theta_1,\theta_2)}(T,x_0,y) }}_{ f_\delta^{\mathcal{H}}(y)}, \end{split}
\end{equation}
where $C_{\theta,x_0,T}$ is the normalizing constant for the density $h^{(\theta_1,\theta_2)}$ and $M_B \leq \frac{ \|b\|^2_\infty T}{2 \delta}$ is an upper bound for $B(y)-B(x)-\frac{(y-x)^2\delta}{2 T}$. Indeed the parameter $\delta$ is introduced 
to control the possibly unbounded term $e^{B(y)-B(x)}$ with $e^{-\frac{(y-x)^2}{2 T}\delta}$. Moreover an appropriate choice of $\delta$, for each specific case, can make the bound sharper.\\
One obtains a decomposition of the density $y \mapsto q^{(\theta_1,\theta_2)}(t,T,x_1,x_2,y)$ as
\begin{equation}\label{GRSq}
\begin{split} 
\frac{q^{(\theta_1,\theta_2)}(t,T,x_1,x_2,y)}{q_0(t,T,x_1,x_2,y)}
& ={\underbrace{ \frac{ \left(C\right)^2 }{\left(1-e^{-\frac{2 z^2}{t}}\right)\left(1-e^{-\frac{2 z^2}{T-t}}\right)}\frac{1}{v^{(\theta_1,\theta_2)}(T,x_1,x_2)}}_{ C^{\mathcal{B}}}} \ 
\underbrace{\bar v^{(\theta_1,\theta_2)}(t,x_1,y) \ \bar v^{(\theta_1,\theta_2)}(T-t,y,x_2)}_{f^{\mathcal{B}}_{x_1,x_2}(y)}.
\end{split}
\end{equation}
\begin{rmrk} The probability to accept a simulation from the instrumental density is respectively the inverse of $\mathcal{C}^\mathcal{H}$ and $\mathcal{C}^{\mathcal{B}}$. Therefore, it is of great importance to get a bound in Proposition~\ref{boundtheta} and for $M_B$ as small as possible.
\end{rmrk}
Let us show how to control the (finite sum) approximation of $f^{\mathcal{B}}_{x_1,x_2}$ and its rate of convergence, as well as the respective quantities for $f^{\mathcal{H}}_\delta$.
%
\begin{lmm} \label{inputGRS}
There exists a sequence of finite series $(f^{\mathcal{B}}_N)_N$ (resp. $(f^{\mathcal{H}}_N)_N$) that converges pointwise to $f^{\mathcal{B}}_{x_1,x_2}$ (resp. to $f^{\mathcal{H}}_\delta$) for $N\to \infty$ with an exponential rate of convergence. 
\end{lmm}
\begin{proof}
One has to find functions $f^{\mathcal{H}}_N$ (resp. $f^{\mathcal{B}}_N$) and an exponentially decreasing sequence $\frak{R}_N^{\mathcal{H}}$ (resp. $\frak{R}_N^{\mathcal{B}}$) such that, for all $y\in \mathbb{R}$, $\left|f^{\mathcal{H}}_\delta(y) - f^{\mathcal{H}}_N(y)\right|\leq \frak{R}_N^{\mathcal{H}}$ (resp. $\left|f^{\mathcal{B}}_{x_1,x_2}(y) - f^{\mathcal{B}}_N(y)\right|\leq \frak{R}_N^{\mathcal{B}}$).\\
Proposition $\ref{boundtheta}$ yields to the following choices
\begin{equation} \label{f:rejbridges}
\begin{cases}
f^{\mathcal{B}}_N := \bar v_N^{(\theta_1,\theta_2)}(t,x_1,y) \cdot  \bar v_N^{(\theta_1,\theta_2)}(T-t,y,x_2) \\
\frak{R}_N^{\mathcal{B}} :=  e^{- \frac{2 z^2}{t}(N+1)} +e^{- \frac{2 z^2}{T-t} (N+1)} -  e^{-2 z^2 \left(\frac{1}{t}+\frac{1}{T-t}\right)(N+1)}
\end{cases}
\end{equation}
and
\begin{equation}\label{f:rejh}\begin{cases}
 f^{\mathcal{H}}_N = \disp \bar v_N^{(\theta_1,\theta_2)}(T,x_0,y) \ \exp{\Big(B(y)-B(x_0)-\frac{(y-x_2)^2}{2T} \delta - M_B\Big)}\\
 \frak{R}_N^{\mathcal{H}} = \disp \frak{R}_N \bar v^{(\theta_1,\theta_2)}(T).
\end{cases}\end{equation}
where $\bar v_N$ and $\frak{R}_N \bar v $ are given in \eqref{eq:renormalizedseries}.\\
Notice that the series $f_N^{\mathcal{H}}$ contains $N+1$ terms, but $f^{\mathcal{B}}_N$ contains $(N+1)^2$ terms, since it is the product of two truncated series, each one with $N+1$ terms. Indeed, if $R \bar v_N^{(\theta_1,\theta_2)}$ is the remainder of $\bar v_N^{(\theta_1,\theta_2)}$, one can write
\[\begin{split} 
f^{\mathcal{B}}_{x_1,x_2} & = \left(\bar v_N^{(\theta_1,\theta_2)}(t,x_1,y) +  R \bar v_N^{(\theta_1,\theta_2)}(t,x_1,y)\right)
\left( \bar v_N^{(\theta_1,\theta_2)}(T-t,y,x_2) +  R \bar v_N^{(\theta_1,\theta_2)}(T-t,y,x_2)\right)\\
& = \left(\bar v_N^{(\theta_1,\theta_2)}(t,x_1,y) \cdot  \bar v_N^{(\theta_1,\theta_2)}(T-t,y,x_2) \right) + \left( \bar v_N^{(\theta_1,\theta_2)}(t,x_1,y) \cdot  R \bar v_N^{(\theta_1,\theta_2)}(T-t,y,x_2) + \right.\\
& \left. \quad + \bar v_N^{(\theta_1,\theta_2)}(T-t,y,x_2) \cdot R \bar v_N^{(\theta_1,\theta_2)}(t,x_1,y) + R \bar v_N^{(\theta_1,\theta_2)}(t,x_1,y) \cdot  R \bar v_N^{(\theta_1,\theta_2)}(T-t,y,x_2)\right).
\end{split}\]
\end{proof}


\section{Numerical simulations}
In this section we provide some numerical simulations for $\mathbb{P}_b$, solution to (\ref{sdeb}). We also measure the performance of the exact simulation method in comparison with the classical Euler-Maruyama method, which has the disadvantage to {\it smooth} the effect of the discontinuous drift.

The section is organized as follows. First we give the pseudo-code for the generalized rejection sampling method adapted for densities which are series whose remainder admits an eventually monotonically decreasing bound.
This pseudo-code is the one used for sampling from the densities $h^{(\theta_1,\theta_2)}$ and $q^{(\theta_1,\theta_2)}(t,T,x_1,x_2,\cdot)$ given in  \eqref{qtheta}.
We devote the second part to the pseudo-code of the retrospective rejection sampling (see Section \ref{EAscheme}), which relies on the simulation from the finite-dimensional distributions of the instrumental measure $\mathfrak{Q}$ given by (\ref{limitsklaw}). In order to minimize the necessary CPU time we propose two different ways to apply the scheme. In the last part we illustrate our simulations.


\subsection{The generalized rejection sampling for sampling under $\mathfrak{Q}$ (GRS)}

In Algorithm \ref{GRS} we give the pseudo-code of the generalized rejection sampling (GRS) method presented in \cite{DMR}, which allows to sample a random variable from an instrumental random variable once the bounded ratio between the densities is an {\it infinite} series whose remainder is eventually monotonically decreasing. If $g(x)$ and $h(x)$ are respectively the instrumental density (w.r.t. the Lebesgue measure) and the density from which one would like to sample, then let us denote by $f(x)$ the function ratio $\disp \frac{1}{m} \frac{h(x)}{g(x)}$, where $m$ is an upper bound of the function $\disp \frac{h}{g}$.

The Algorithm \ref{GRS} requires the following quantities and functionals:
\begin{itemize}
\item {\color{blue}$g$}: instrumental density under which it is known how to sample,
\item {\color{blue}$N_{max}$}: the maximal number of terms of the series one decides to consider,
\item {\color{blue}$(f_N)_{N=0,\ldots,N_{max}}$}: the partial sums of the series $f$,
\item {\color{blue} $(\frak{R}^f_N)_{N=0\ldots,N_{max}}$}: decreasing sequence of bounds for the remainder $f-f_N$,
\item {\color{blue} $I_\frak{R} ^f$}: a (piecewise constant) non decreasing function $(0,1) \to \{0,\ldots,N_{max}\}$, inverse of $\frak{R}^f_N$:\\ $I_\frak{R} ^f(u) = \inf \{N \leq N_{max}: \frak{R}^f_N \leq u\}$,
\end{itemize}

\begin{algorithm}[H]
\SetKwInOut{Input}{Input}
\SetKwInOut{Output}{Output}
\caption{The generalized rejection sampling \textbf{GRS} \label{GRS}}
\Input{$g$, $N_{max}$, $(f_N)_N$; $(\frak R^f_N)_N$, $I_\frak{R}^f$.}
\Output{$x$,exact: $x$ is a sample from the desired density 
and exact is True if the simulation is exact or False if it is not exact.}

reject $\leftarrow$ True\;
\While{\emph{reject}}{
	sample a standard uniform $u$\;
	sample $y$ from $g$\;
	$N \leftarrow 0$\;
	\While{ $|f_N(y)-u|<\frak{R}^f_N$ and $N<N_{max}$}{
		$N \leftarrow I_\frak{R}^f(|f_N(y)-u|)$\;
	}
	\If{$f_N>u$}{
		reject $\leftarrow$ False \;
		x $\leftarrow$ y and exact $\leftarrow$ True\;
	}
	\If{$N=N_{max}$}{
		reject $\leftarrow$ False\;
		x $\leftarrow$ y and exact  $\leftarrow$ False\;
	}
}		
\Return $x$ and exact
\end{algorithm}

\vspace{1cm}

To obtain the exact simulation, that is the acceptance or rejection of each sample, one may have to consider a big number of terms of the series $f_N$. Therefore we decide to fix the integer $N_{max}=I_\frak{R}^f(0.00005)$, i.e. the smallest integer $N$ such that $2 \frak{R}^f_{N}$ is smaller than 0.0001. Indeed the probability that for all $N \leq N_{max}$ one has not been able to accept or reject a sample is smaller than twice the bound $\frak{R}^f_{N_{max}}$. 
All sequences $(\frak{R}^f_{N})_N$ considered in this document are exponentially decreasing and all simulations turn out to be exact. Moreover the procedure to accept or reject is really fast since usually it is done by computing only $f_0,f_1$ and $f_2$.

Clearly, the \textbf{GRS} enables us to sample from the finite-dimensional distributions of $\mathfrak{Q}$ (\ref{limitsklaw}). In particular we will use it to sample from the densities $h^{(\theta_1,\theta_2)}(T,x_0,\cdot)$ and $q^{(\theta_1,\theta_2)}(t,T,x_1,x_2,\cdot)$ in (\ref{qtheta}) respectively with instrumental densities the two Gaussian densities $p_0\left(\frac{T}{1-\delta},x_0,\cdot\right)$ and $q_0(t,T,x_1,x_2,\cdot)$.\\
The sequence $(f_N, \ \frak{R}^f_N)_{N\leq N_{max}}$ for $h^{(\theta_1,\theta_2)}$ and $q^{(\theta_1,\theta_2)}$ are provided respectively by \eqref{f:rejh} and \eqref{f:rejbridges}, where $v^{(\theta_1,\theta_2)}_{N} (t,x,y)$ is the truncation at the $N-$th term of the series given by (\ref{vtheta}) and (\ref{eq:seriestheta}). The piecewise constant functions $I_\frak{R}^f$ are easy to compute explicitly since the bounds $\frak{R}^f_N$ decrease exponentially.

\begin{rmrk} Extending Remark~\ref{termbound}, the \textbf{GRS} can be used also to simulate the skew Brownian motions with drift and with two semipermeable barriers.
\end{rmrk}

\subsection{The retrospective rejection sampling}
%

The following algorithm describes the scheme that returns the pair $(T+t_0, X_{T})$. $X_{T}$ is a sample of the solution $\mathbb{P}_b$ on $[t_0,T+t_0]$ of (\ref{sdeb}) at the final time $T+t_0$. The sample is obtained through retrospective rejection sampling from the finite-dimensional distributions of the instrumental measure $\mathfrak{Q}$. The algorithm recalls the external function \textbf{GRS}.

The algorithm needs some parameters derived from the given drift $b$, such as the (half) heights $\theta_1,\theta_2$ of the two jumps, the bounded non negative function $\phi^+_b$ defined in (\ref{phib}) and its sup norm $\|\phi^+_b\|_\infty$.

\begin{algorithm}[H]
\caption{The retrospective rejection sampling \textbf{RRS}}
\SetKwInOut{Input}{Input}\SetKwInOut{Output}{Output}
\Input{The starting point $x_0$ (at the time $t_0$), the time increment $T$;}
\Output{A sample of $X_{T}$ at the final time $T+t_0$.}
reject $\leftarrow$ True\;
\While{\emph{reject}}{
	reject $\leftarrow$ False\;
	simulate a Poisson Point Process on $[0,T]\times[0,\|\phi_b^+\|_\infty]$: $(\tau_k, x_{\tau_k})_{k=1,\ldots,M}$\;
	sample $X_T$ from the density $h^{(\theta_1,\theta_2)}$ through \textbf{GRS}\;
	$\tau_0 \leftarrow 0$, $y_0 \leftarrow x_0$\;
	\For{$k=1$ \KwTo M }{
		sample $y_k$ through \textbf{GRS} for the bridge density $q^{(\theta_1,\theta_2)}$ connecting $(\tau_{k-1},y_{k-1})$ and $(T,X_T)$ \;
		\If{$\phi_b^+(x_{\tau_k})>y_k$}{
			reject $\leftarrow$ True\;
			exit from this cycle and start again\;
			}
		 }	
}
\Return $X_T$.
\end{algorithm}

\begin{rmrk} The algorithm can actually return the skeleton of the Brownian motion with drift $b$: the vectors $(t_0,\tau_1+t_0,\ldots, \tau_M+t_0,T+t_0), (x_0,y_1,\ldots,y_M,X_T)$. One can then add to the skeleton the simulation at any time instance $t$ in $(t_0,T+t_0)$ using the bridges dynamics.
\end{rmrk}

\begin{rmrk}
It is possible to sample the Poisson point process simulated on the rectangle $[t_0,T+t_0]\times[0,\|\phi_b^+\|_\infty]$ progressively within the rejection procedure. As a consequence one obtains a more efficient algorithm.
\end{rmrk}

In order to make the algorithm more efficient one notices that, if the time increment $T$ is large, there will be a large number of Poisson points which slows down the algorithm since it can reject more often. Using the Markov property of the considered process, one can split the time increment into congruent time intervals of length $t$ smaller than a fixed value $T_{el}$. One then applies the \textbf{RRS} on the different time intervals of length $t$ with new initial conditions given by the ending point on the previous interval. We will call \textbf{SRRS} the split RRS obtained choosing  $T_{el}\leq \|\phi_b^+\|_\infty^{-1}$ in a convenient way in order to minimize the computational times.
If $T_{el}$ were $\|\phi_b^+\|_\infty^{-1}$, the Poisson process on the rectangle in \textbf{RRS} would have intensity 1. Therefore, with probability higher than $e^{-1}$ (when the Poisson process on $[0,t]\times [0,\|\phi_b^+\|_\infty]$ has an empty realization), one avoids to apply \textbf{GRS} for bridges.\\
We are proposing the next pseudo-code whose inputs are the same as for the algorithm \textbf{RRS}.
  
\vspace{0.5cm}

\begin{algorithm}[H]
\caption{The split retrospective rejection sampling \textbf{SRRS} \label{SRRS}}
\SetKwInOut{Input}{Input}\SetKwInOut{Output}{Output}
\Input{The initial conditions ($t_0,\ x_0$),\\
the time increment $T$,\\
the function $\phi^+_b$ and its upper bound $m$ (see (\ref{phib}));}
\Output{A sample of $X_{T}$ at the final time $T+t_0$.} 
Split the interval $[0,T]$ in $m'$ congruent intervals of length smaller than $T_{el}=m^{-1}$\;
\ForAll{$j=1,\ldots,m'$}{
obtain $X_{T_{el}*j}$ through the \textbf{RRS} with input the time interval $T_{el}$ and the initial conditions $(T_{el}*(j-1),x_0)$\;}
\Return $X_T$
\end{algorithm}

\vspace{0.3cm}

We will compare in the next subsection the CPU times using the retrospective rejection sampling described above for some piecewise smooth drift $b$ defined by (\ref{piecewiseconstantdrift}), see Figure \ref{CPUtimes}.\\
We coded in Python and executed the programs on a personal computer equipped with an Intel Core i5 processor, running at 2.5 Ghz.

\subsection{Simulations}
%

Let us consider two examples of Brownian diffusions starting from $x_0=0.5$ whose drift $b_1$ (resp. $b_2$) is piecewise smooth and admits two discontinuity points at $0$ and $z=1$:
\begin{equation}\label{piecewiseconstantdrift}
\bar b_1(x):=
\begin{cases} 
 0  &  x< 0\\
 1 & 0<x< 1\\
 0 & x>1
\end{cases}, \qquad 
\bar b_2(x):=
\begin{cases} 
 - 2 \cos(x)  &  x< 0\\
 \sin(x) & 0<x< 1\\
 \cos(x-z)+sin(z) & x>1
 \end{cases}, 
\end{equation}
and $\bar b_1(0)=\frac12, \bar b_1(1)=\frac12$ as in (\ref{bzi}), and $\theta_1=1/2=-\theta_2$ (resp. $\bar b_2(0)=-1, \bar b_2(1)=\frac12$, and $\theta_1 = 1$ and $\theta_2 = 1/2$). (Figure~\ref{Driftb2} represents the drift $\bar b_2$.)

In each case we need to choose the time length $T_{el}$ (to apply the \textbf{SRRS}, see Algorithm~\ref{SRRS}) and the parameter $\delta \in (0,1)$ appearing in \eqref{GRSh}. In the first case we fixed $T_{el}=0.55$ ($\leq \|\phi^+_{\bar b_1}\|_\infty^{-1}=2$) and $\delta=0.75$. In the secund case we choose $T_{el}=0.2$ and $\delta=0.6$. Let us briefly explain how we took our decision.\\
Once $T_{el}$ has been fixed, we choose $\delta$ such that the quantity $\frac12 \|\bar b_1\|_\infty^2 \frac{T_{el}}\delta$ is as small as possible (hence we can take $M_B$ in \eqref{GRSh} equal to it). Indeed the sharpness of this quantity and of the constant $C$ in \eqref{eq:boundv} determines the probability to accept a sample from the instrumental density as a sample from $h^{(\theta_1,\theta_2)}$. The latter probability is the inverse of the constant $C^{\mathcal{H}}$ given explicitly in \eqref{GRSh}. Let us recall that the constant $C$ is a factor of the quantity $C^{\mathcal{B}}$ (see \eqref{GRSq}) as well, hence determines the probability of accepting a sample of a Brownian bridge as a sample from the desired density.\\
The choice of $T_{el}$ is more delicate and it is based on the computational time of the algorithm \textbf{RRS}. The algorithm \textbf{SRRS} splits the given interval $(0,T)$ into intervals of length between $\frac12 T_{el}$ and $T_{el}$ and applies \textbf{RRS}. We choose $T_{el}$ such that \textbf{RRS} is faster on $( \frac12 T_{el}, T_{el})$.

\begin{figure}
\begin{center}
\subfigure[$X_T$ is the Brownian diffusion with drift $\bar b_1$ at time $T$]{\includegraphics[width=8cm]{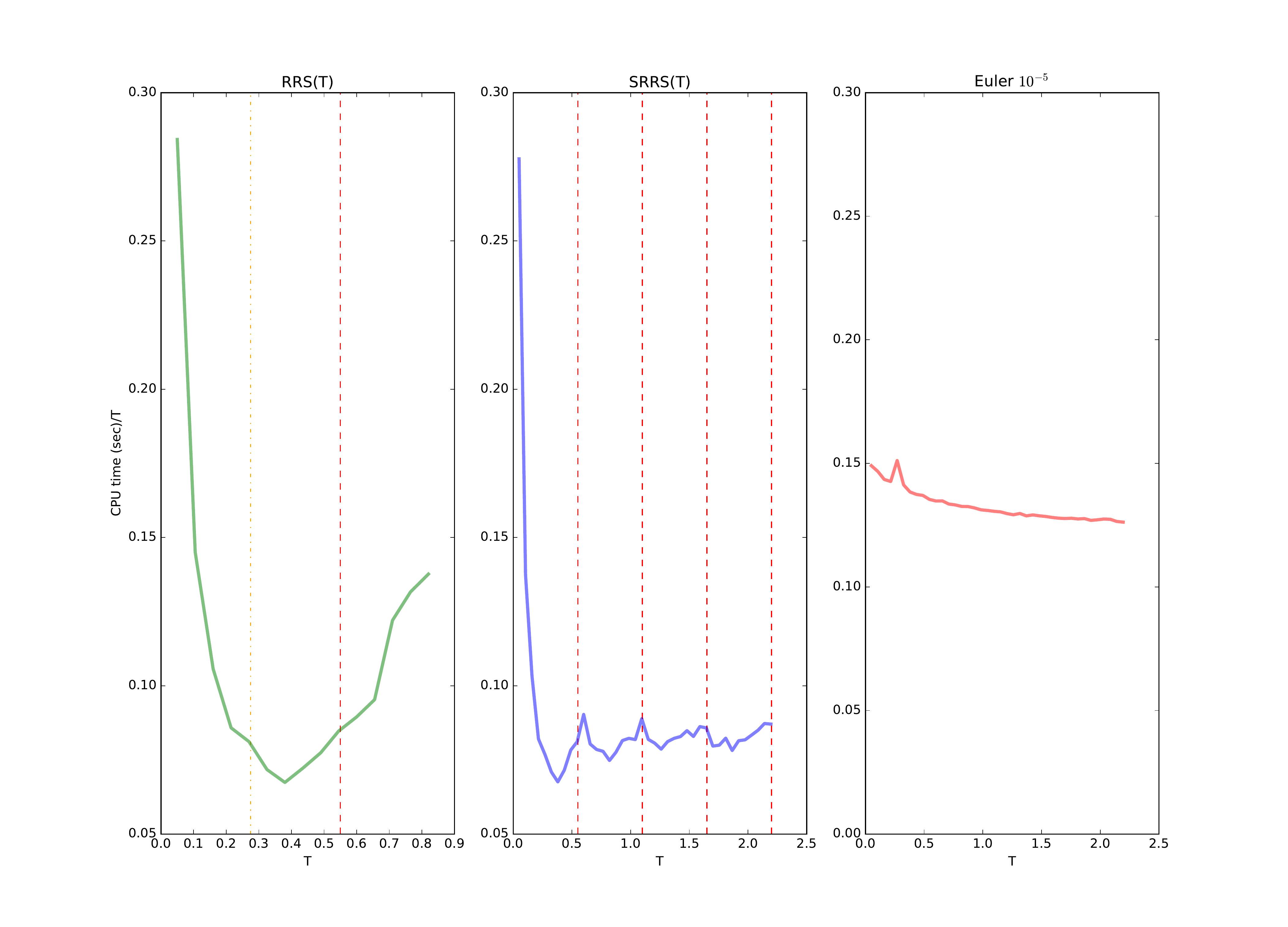}}
\subfigure[$X_T$ is the Brownian diffusion with drift $\bar b_2$ at time $T$]{\includegraphics[width=8.1cm]{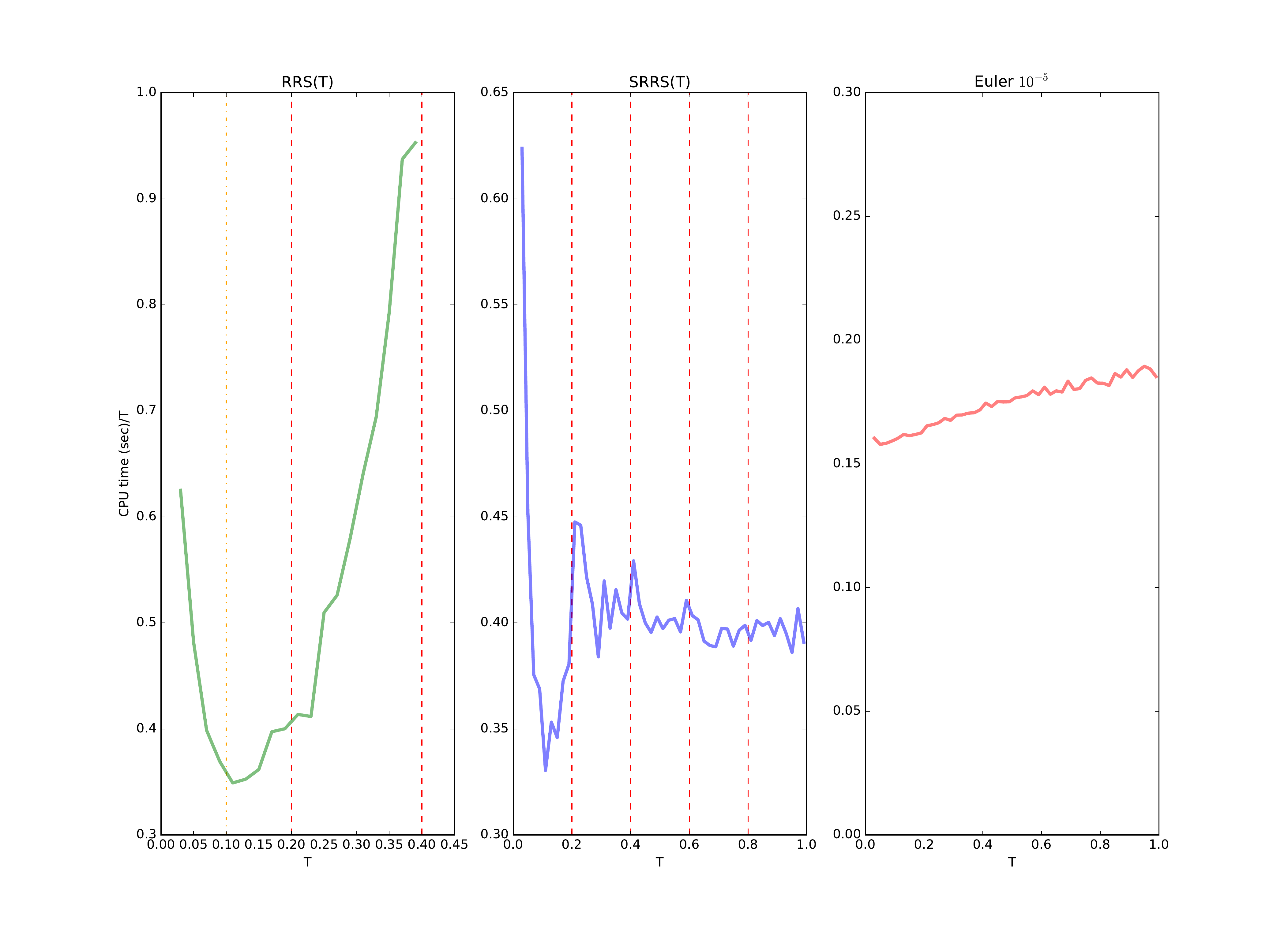}}
\end{center}
\caption{If $X_T$ is the Brownian diffusion with drift $\bar b$ at time $T$, then the curves represent the ratio between the CPU times to sample $X_T$ and the time $T$ with the different methods. The (red) dashed lines represent the multiples of the maximal length for a split interval in the algorithm SRRS ($T_{el} \leq \|\phi^+_{\bar b}\|_\infty^{-1}$). The curve obtained through the RRS method reaches its minimum between the first two dashed lines (orange and red). We considered 1000 simulations and computed the average time.}\label{CPUtimes}
\end{figure}

The CPU time for the RRS does not grow linearly with the time, as one can easily notice from Figure~\ref{CPUtimes}. The Euler-Maruyama method and SRRS instead show an asymptotic linear growth of the CPU time as function of the time $T$. Sometimes the growth factor of SRRS is considerably faster, and sometimes slower as in the cases of the drift $\bar b_2$. This is due to the nature of the drifts and to the quality of the bounds.

\begin{figure}
\begin{center}
\subfigure[The density of the Brownian diffusion $\mathbb{P}_{\bar b_1}$ at time 1. Bandwidth=0.1. $T_{el}=0.55, \ \delta=0.75$.]{\includegraphics[width=8cm]{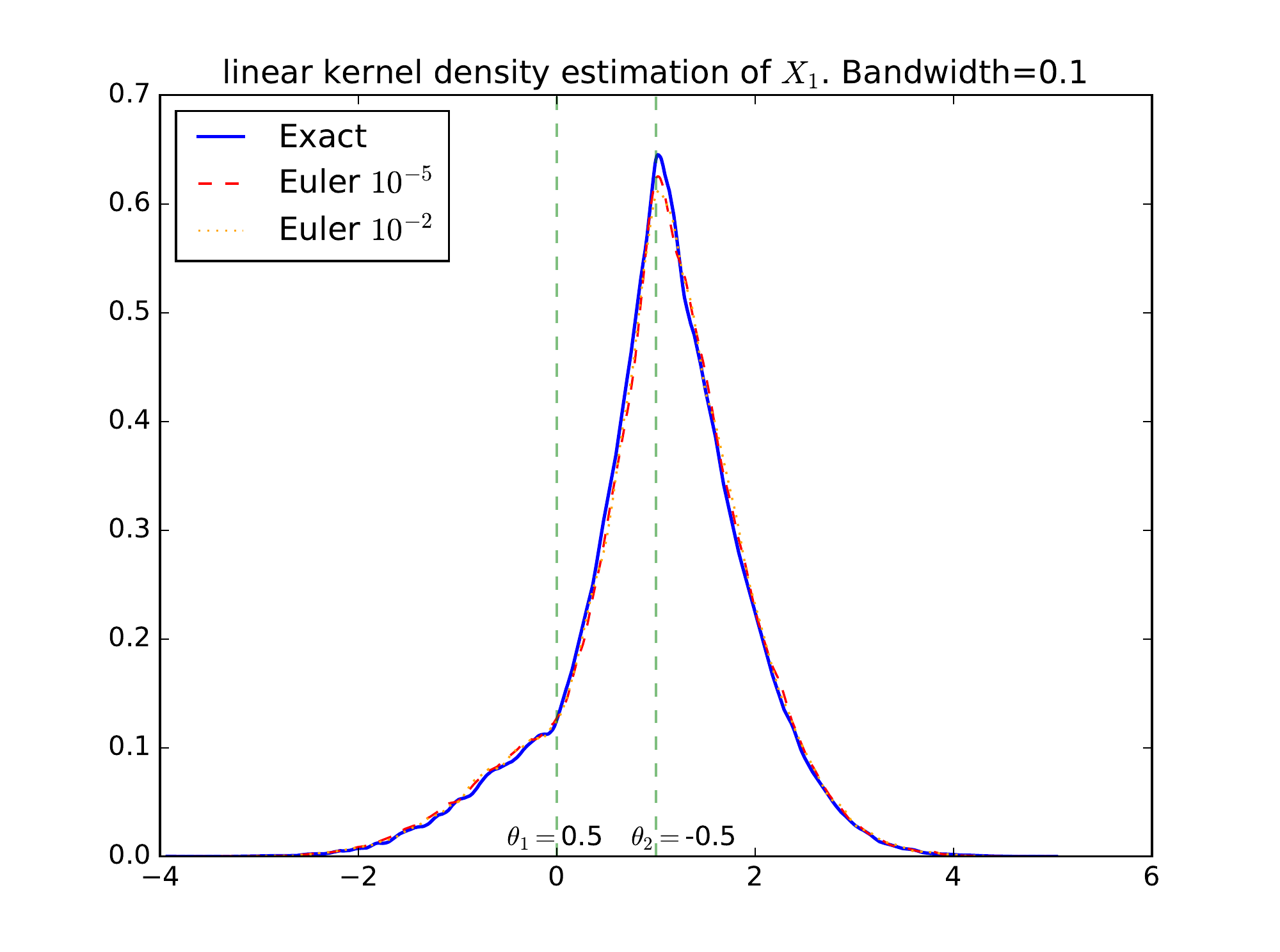}}
\subfigure[The density of the Brownian diffusion $\mathbb{P}_{\bar b_2}$ at time $0.6$. $T_{el}=0.2$, $\delta=0.6$.]{\includegraphics[width=8cm]{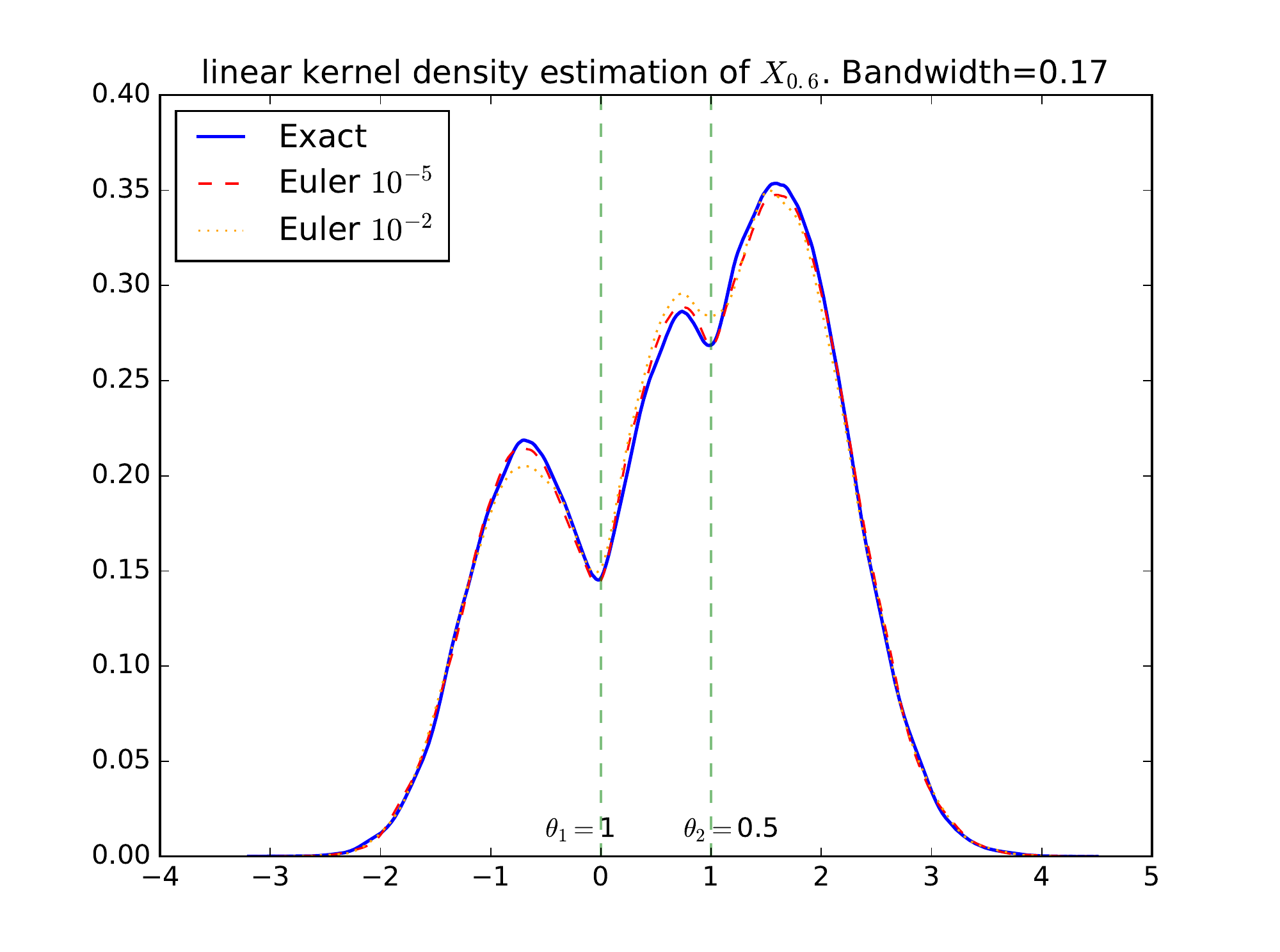}}
\end{center}
\caption{The kernel density estimations of the density obtained from $10^5$ samples with the Euler-Maruyama scheme ($10^{-5}$ dashed line, and $10^{-2}$ dotted line) and the exact \textbf{SRRS} algorithm (with $T_{el}$ and $\delta$ chosen above). The dashed vertical lines represent the points of discontinuity of the drift. The process has fixed initial condition $X_0=0.5.$}\label{ComparisonEuler}
\end{figure}

Figure~\ref{ComparisonEuler} shows that the Euler-Maruyama scheme needs a very small discretization step to be more precise at the discontinuities of the drift. \\
Considering the drift $\bar b_1$ which is a indicator function, our simulation of the $10^5$ samples is exact and it is much faster than Euler-Maruyama with discretization step $10^{-5}$. The average time for a single sample is respectively of 0.09 and 0.18 seconds. The Euler-Maruyama scheme with step $10^{-2}$ instead is faster ($0.0002$ seconds). \\
In the other considered case, the Euler-Maruyama method is faster. For the drift $\bar b_2$, the average CPU times for one simulation of $X_{0.6}$, is $0.29 \,s $  and $0.1 \, s$ for the exact simulation through the \textbf{SRRS} and the finest Euler-Maruyama respectively. To obtain $X_{0.6}$ the \textbf{SRRS} has applied three times \textbf{RRS} to intervals of length $0.2$. 

Let us quantify the average number of rejections for a single sample from the densities $h^{(\theta_1,\theta_2)}(y)$ (see \eqref{GRSh}). For $\bar b_1$ with $T=T_{el}=0.55$, it is around $5$; indeed the probability of accepting a sample is the inverse of the constant $C^{\mathcal{H}}$ and it is about 0.196.  
For the densities $q^{(\theta_1,\theta_2)}(t,T,x_0,x_2,\cdot)$, $t<T$, the average number of rejections is around 2.
Moreover, as expected, the number of terms of the series necessary to the decision (accept or reject) its rarely bigger than 1 and the average is around 1. This holds also in the case of the drift $\bar b_2$.
Finally the average number of path rejection for each RRS in the SRRS is slightly larger than 5.

Let us consider the drift $\bar b_2$. In this case the average number of rejections for a single sample from the densities $h^{(\theta_1,\theta_2)}(y)$ and $q^{(\theta_1,\theta_2)}(t,T,x_0,x_2,\cdot)$, $t<T$ are respectively around $7$ and $6$. Finally the average number of path rejection for each RRS in the SRRS is slightly bigger than 9.

Figure~\ref{pathStep} shows the realization of a path of the Brownian diffusion with drift $\bar b_1$. The path has been simulated as follows: given the skeleton with SRRS, it is sampled at each time of the discretization of step $10^{-3}$ following the bridges dynamics.\\


\begin{figure}[t]\centering
\begin{minipage}{0.45\textwidth}
\centering
\includegraphics[width=8cm]{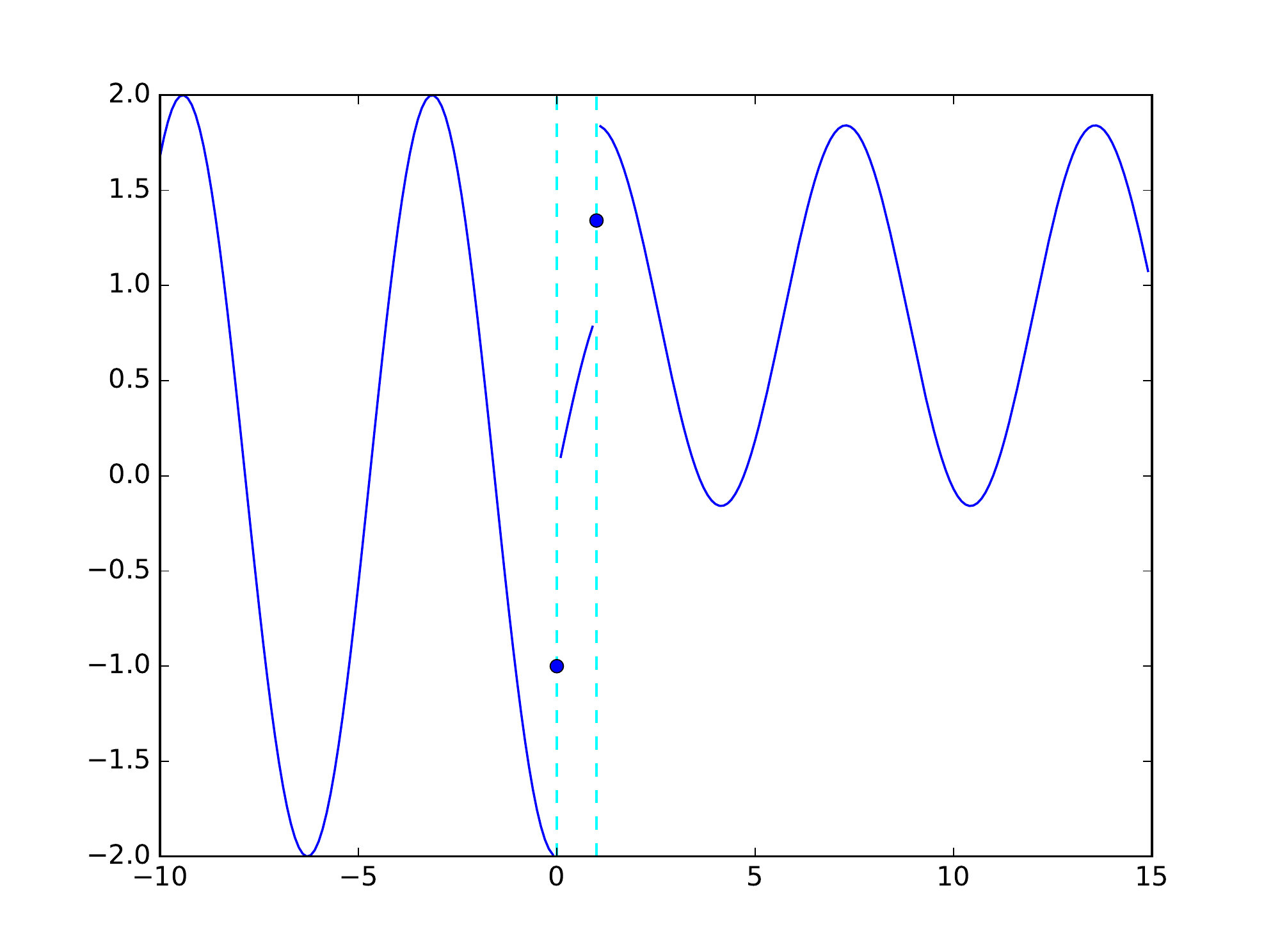}
\caption{The drift $\bar b_2$.}\label{Driftb2}
\end{minipage}
\hfill
\begin{minipage}{0.45\textwidth}
\includegraphics[width=8cm]{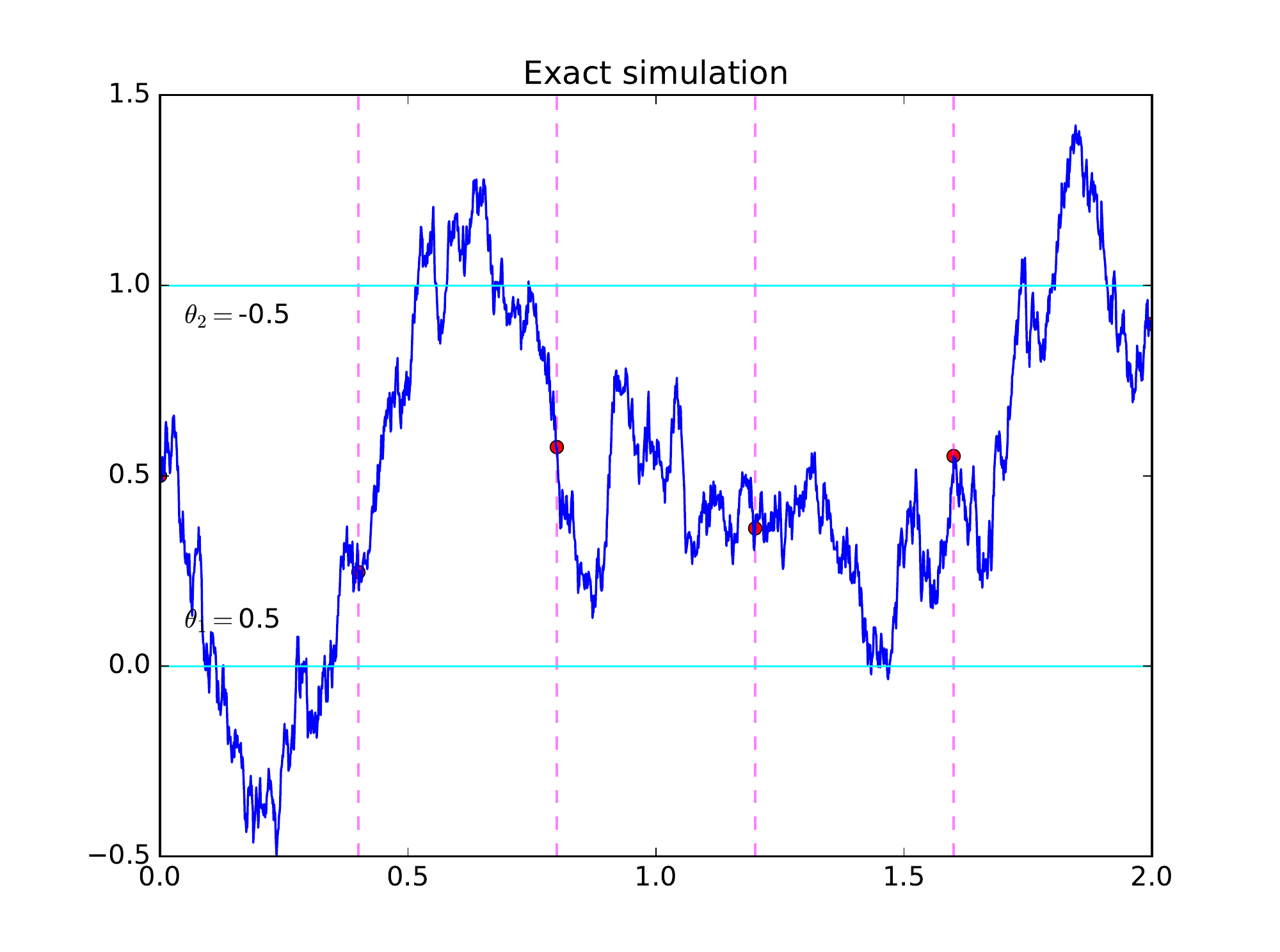}
\caption{A realization of a path $t \mapsto X_t$ up to time $T=2$ under $\mathbb{P}_{\bar b_1}$. The horizontal lines represent the points of discontinuity of the drift. The dots on the path are the skeleton of the process obtained by \textbf{SRRS} which splits $[0,2]$ in five subintervals delimited by the dotted vertical lines.} \label{pathStep}
\end{minipage}
\end{figure}


\vspace{1cm}
\textbf{Acknowledgements:}
The authors acknowledge the Deutsch-Franz\"{o}sische Hochschule - Universit\'e Franco-Allemande (DFH-UFA) and the Research Training Group 1845 {\it Stochastic Analysis with Applications in Biology, Finance and Physics} for their financial support.\\
The authors would like to warmly thank Pierre \'Etor\'e, Antoine Lejay and Lionel Len\^otre for fruitful discussions on the topic treated here.

\addcontentsline{toc}{section}{Bibliography}

\bibliographystyle{plain}
\bibliography{exact_algorithm_bibliography}

\end{document}